\documentclass[a4paper,12pt]{article}
\usepackage{latexsym}
\usepackage{amsmath}
\usepackage{amssymb}
\usepackage{enumerate}
\usepackage{bm}
\usepackage{mathrsfs}
\usepackage{color}

\definecolor{refkey}{gray}{0.5}
\definecolor{labelkey}{gray}{0.2}

\usepackage{theorem}
\newtheorem{theorem}{Theorem}[section]
\newtheorem{proposition}[theorem]{Proposition}
\newtheorem{lemma}[theorem]{Lemma}
\newtheorem{corollary}[theorem]{Corollary}

\newtheorem{thm}{Theorem}

\theorembodyfont{\rmfamily}
\newtheorem{proof}{\textmd{\textit{Proof.}}}

\newtheorem{remark}[theorem]{Remark}

\newtheorem{definition}[theorem]{Definition}

\makeatletter

\@addtoreset{equation}{section}
\makeatother

\newcommand{\qedd}{\hfill \Box}
\newcommand{\ve}{\varepsilon}
\newcommand{\del}{\partial}
\newcommand{\lra}{\longrightarrow}
\newcommand{\e}{\mathrm{e}}

\newcommand{\N}{\ensuremath{\mathbb{N}}}

\newcommand{\R}{\ensuremath{\mathbb{R}}}
\newcommand{\Sph}{\ensuremath{\mathbb{S}}}
\newcommand{\cC}{\ensuremath{\mathcal{C}}}
\newcommand{\cE}{\ensuremath{\mathcal{E}}}
\newcommand{\cI}{\ensuremath{\mathcal{I}}}
\newcommand{\cL}{\ensuremath{\mathcal{L}}}
\newcommand{\scN}{\ensuremath{\mathscr{N}}}
\newcommand{\fm}{\ensuremath{\mathfrak{m}}}
\newcommand{\bS}{\ensuremath{\mathbf{S}}}
\newcommand{\sC}{\ensuremath{\mathsf{C}}}
\newcommand{\sS}{\ensuremath{\mathsf{S}}}

\def\vol{\mathop{\mathrm{vol}}\nolimits}
\def\diam{\mathop{\mathrm{diam}}\nolimits}
\def\div{\mathop{\mathrm{div}}\nolimits}

\def\loc{\mathop{\mathrm{loc}}\nolimits}

\def\Var{\mathop{\mathrm{Var}}\nolimits}
\def\Ric{\mathop{\mathrm{Ric}}\nolimits}
\def\BE{\mathop{\mathrm{BE}}\nolimits}
\def\CD{\mathop{\mathrm{CD}}\nolimits}
\def\RCD{\mathop{\mathrm{RCD}}\nolimits}
\def\HS{\mathop{\mathrm{HS}}\nolimits}

\newcommand{\Grad}{\bm{\nabla}}
\newcommand{\Lap}{\bm{\Delta}}

\newcommand{\rev}[1]{\overleftarrow{#1}}

\title{A semigroup approach to Finsler geometry:\\ Bakry--Ledoux's isoperimetric inequality}
\author{Shin-ichi Ohta\thanks{Department of Mathematics, Osaka University,
Osaka 560-0043, Japan ({\sf s.ohta@math.sci.osaka-u.ac.jp})}
\thanks{RIKEN Center for Advanced Intelligence Project (AIP),
1-4-1 Nihonbashi, Tokyo 103-0027, Japan}
\thanks{Supported in part by JSPS Grant-in-Aid for Scientific Research (KAKENHI)
15K04844, 19H01786.}}

\date{\today}
\begin{document}

\maketitle

\begin{abstract}
We develop the celebrated semigroup approach \`a la Bakry et al on Finsler manifolds,
where natural Laplacian and heat semigroup are nonlinear,
based on the Bochner--Weitzenb\"ock formula established by Sturm and the author.
We show the $L^1$-gradient estimate on Finsler manifolds
(under some additional assumptions in the noncompact case),
which is equivalent to a lower weighted Ricci curvature bound and the improved Bochner inequality.
As a geometric application, we prove Bakry--Ledoux's Gaussian isoperimetric inequality,
again under some additional assumptions in the noncompact case.
This extends Cavalletti--Mondino's inequality on reversible Finsler manifolds to non-reversible metrics,
and improves the author's previous estimate,
both based on the localization (also called needle decomposition) method.
\smallskip

\end{abstract}

\renewcommand{\contentsname}{{\large Contents}}
{\small
\tableofcontents
}

\section{Introduction}%%%%%%%%%%%%%%%%%%%%%%
%%%%%%%%%%%%%%%%

The aim of this article is to put forward the semigroup approach
in geometric analysis on Finsler manifolds,
based on the Bochner--Weitzenb\"ock formula established in \cite{OSbw}.
There are already a number of applications of the Bochner--Weitzenb\"ock formula
(including \cite{WX,Xi,YH,Oineq}),
and the machinery in this article would contribute to a further development.
In addition, our treatment of a nonlinear generator and the associated nonlinear semigroup
(Laplacian and heat semigroup) could be of independent interest from the analytic viewpoint.

The celebrated theory developed by Bakry, \'Emery, Ledoux et al
(called the \emph{$\Gamma$-calculus}) studies symmetric generators
and the associated linear, symmetric diffusion semigroups under a kind of Bochner inequality
(called the (\emph{analytic}) \emph{curvature-dimension condition}).
Attributed to Bakry--\'Emery's original work \cite{BE},
this condition will be denoted by $\BE(K,N)$ in this introduction,
where $K \in \R$ and $N \in (1,\infty]$ are parameters corresponding to
`curvature' and `dimension', respectively.
This technique is extremely powerful in studying various inequalities
(log-Sobolev and Poincar\'e inequalities, gradient estimates, etc.) in a unified way,
we refer to \cite{BE} and the recent book \cite{BGL} for more on this theory.

On a Riemannian manifold equipped with the Laplacian $\Delta$,
$\BE(K,N)$ means the following Bochner-type inequality:
\[ \Delta \bigg[ \frac{\|\nabla u\|^2}{2} \bigg] -\langle \nabla(\Delta u),\nabla u \rangle
 \ge K\|\nabla u\|^2 +\frac{(\Delta u)^2}{N}. \]
Thereby a Riemannian manifold with Ricci curvature not less than $K$
and dimension not greater than $N$ (more generally,
a weighted Riemannian manifold of weighted Ricci curvature $\Ric_N \ge K$)
is a fundamental example satisfying $\BE(K,N)$.

Later, inspired by \cite{CMS,OV},
Sturm~\cite{vRS,StI,StII} and Lott--Villani~\cite{LV} introduced
the (\emph{geometric}) \emph{curvature-dimension condition} $\CD(K,N)$
for metric measure spaces in terms of optimal transport theory.
The condition $\CD(K,N)$ characterizes $\Ric \ge K$ and $\dim \le N$ (or $\Ric_N \ge K$)
for (weighted) Riemannian manifolds,
and its formulation requires a lower regularity of spaces than $\BE(K,N)$.
We refer to Villani's book \cite{Vi} for more on this rapidly developing theory.
It was shown in \cite{Oint} that $\CD(K,N)$ also holds
and characterizes $\Ric_N \ge K$ for Finsler manifolds,
where the natural Laplacian and the associated heat semigroup are nonlinear.
For this reason, Ambrosio, Gigli and Savar\'e~\cite{AGSrcd} introduced
a reinforced version $\RCD(K,\infty)$ called the \emph{Riemannian curvature-dimension condition}
as the combination of $\CD(K,\infty)$ and the linearity of heat semigroup,
followed by the finite-dimensional analogue $\RCD^*(K,N)$
investigated by Erbar, Kuwada and Sturm~\cite{EKS} (see also \cite{Gi1,Gi2}).
It then turned out that $\RCD^*(K,N)$ is equivalent to $\BE(K,N)$ (\cite{AGSboc,EKS}),
this equivalence justifies the term `curvature-dimension condition'
which actually came from the similarity to Bakry's theory.

In this article, we develop the theory of Bakry et al on Finsler manifolds.
We consider a Finsler manifold $M$ equipped with a Finsler metric $F:TM \lra [0,\infty)$
and a positive $\cC^{\infty}$-measure $\fm$ on $M$.
We will not assume that $F$ is \emph{reversible}, thereby $F(-v) \neq F(v)$ is allowed.
The key ingredient, the \emph{Bochner inequality} under $\Ric_N \ge K$,
was established in \cite{OSbw} as follows:
\begin{equation}\label{eq:B}
\Delta\!^{\Grad u} \bigg[ \frac{F^2(\Grad u)}{2} \bigg] -d(\Lap u)(\Grad u)
 \ge KF^2(\Grad u) +\frac{(\Lap u)^2}{N}.
\end{equation}
This Bochner inequality has the same form as the Riemanian case
by means of the mixture of the nonlinear Laplacian $\Lap$ and its linearization $\Delta\!^{\Grad u}$.
Despite of this mixture, we could derive Bakry--\'Emery's $L^2$-gradient estimate
as well as Li--Yau's estimates on compact manifolds (see \cite[\S 4]{OSbw}).
We proceed further in this direction and show the \emph{improved Bochner inequality}
under $\Ric_{\infty} \ge K$ (Proposition~\ref{pr:Boc+}):
\begin{equation}\label{eq:B+}
\Delta\!^{\Grad u} \bigg[ \frac{F^2(\Grad u)}{2} \bigg] -d(\Lap u)(\Grad u) -KF^2(\Grad u)
 \ge d[F(\Grad u)] \big( \nabla^{\Grad u} [F(\Grad u)] \big).
\end{equation}
The first application of \eqref{eq:B+} is the \emph{$L^1$-gradient estimate} (Theorem~\ref{th:L1}),
where we include also the noncompact case but with some additional (likely redundant) assumptions,
see the theorem below where we assume the same conditions.
We also see that the Bochner inequalities \eqref{eq:B} (with $N=\infty$), \eqref{eq:B+}
and the $L^2$- and $L^1$-gradient estimates are all equivalent to $\Ric_{\infty} \ge K$
(Theorem~\ref{th:char}).

The second, geometric application of \eqref{eq:B+} is a generalization of
\emph{Bakry--Ledoux's Gaussian isoperimetric inequality} (Theorem~\ref{th:BL}):

\begin{thm}[Bakry--Ledoux's isoperimetric inequality]
Let $(M,F,\fm)$ be complete and satisfy $\Ric_{\infty} \ge K>0$, $\fm(M)=1$,
$\sC_F<\infty$ and $\sS_F<\infty$.
We also assume that
\[ d[F(\Grad u_t)] \big( \nabla^{\Grad u_t} [F(\Grad u_t)] \big) \,\in L^1(M) \]
holds for any global solution $(u_t)_{t \ge 0}$ to the heat equation with $u_0 \in \cC^{\infty}_c(M)$
and any $t>0$.
Then we have
\begin{equation}\label{eq:BL}
\cI_{(M,F,\fm)}(\theta) \ge \cI_K(\theta)
\end{equation}
for all $\theta \in [0,1]$, where
\[ \cI_K(\theta):=\sqrt{\frac{K}{2\pi}} \e^{-Kc^2(\theta)/2} \qquad
 \text{with}\ \ \theta=\int_{-\infty}^{c(\theta)} \sqrt{\frac{K}{2\pi}} \e^{-Ka^2/2} \,da. \]
\end{thm}

Here $\cI_{(M,F,\fm)}:[0,1] \lra [0,\infty)$ is the \emph{isoperimetric profile}
defined as the least boundary area of sets $A \subset M$ with $\fm(A)=\theta$
(see the beginning of Section~\ref{sc:BL}), and
$\sC_F$ (resp.\ $\sS_F$) is the ($2$-)\emph{uniform convexity} (resp.\ \emph{smoothness})
\emph{constant} which bounds the \emph{reversibility},
\begin{equation}\label{eq:revF}
\Lambda_F:=\sup_{v \in TM \setminus 0} \frac{F(v)}{F(-v)} \,\in [1,\infty],
\end{equation}
as $\Lambda_F \le \min\{\sqrt{\sC_F}, \sqrt{\sS_F} \}$ (see Lemma~\ref{lm:rev}).
(In particular, the forward completeness is equivalent to the backward completeness,
and we denoted it by the plain \emph{completeness} in the theorem.)
All the conditions $\sC_F<\infty$, $\sS_F<\infty$, and 
$d[F(\Grad u_t)](\nabla^{\Grad u_t} [F(\Grad u_t)]) \in L^1(M)$
hold true in the compact case.
In the noncompact case, however, there are technical difficulties and
it is unclear how to remove them in this semigroup approach (see \S\ref{ssc:hypo} for a discussion).
We remark that, in \cite{Oneedle} based on the needle decomposition,
we did not need those conditions.

The inequality \eqref{eq:BL} has the same form as the Riemannian case in \cite{BL},
and it is sharp and the model space is the real line $\R$ equipped with
the normal (Gaussian) distribution $d\fm=\sqrt{K/2\pi} \, \e^{-Kx^2/2} \,dx$.
See \cite{BL} for the original work of Bakry and Ledoux on linear diffusion semigroups
(influenced by Bobkov's works \cite{Bo1,Bo2}),
and \cite{Bor,SC} for the classical Euclidean or Hilbert cases.
We also refer to \cite{AM} for the Gaussian isoperimetric inequality on $\RCD(K,\infty)$-spaces
by a refinement of the $\Gamma$-calculus.

The above theorem extends Cavalletti--Mondino's isoperimetric inequality in \cite{CM}
to non-reversible Finsler manifolds.
Precisely, they considered essentially non-branching metric measure spaces $(X,d,\fm)$
satisfying $\CD(K,N)$ for $K \in \R$ and $N \in (1,\infty)$,
and showed the sharp \emph{L\'evy--Gromov type isoperimetric inequality} of the form
\[ \cI_{(X,d,\fm)}(\theta) \ge \cI_{K,N,D}(\theta) \]
with $\diam X \le D \ (\le \infty)$.
The case of $N=\infty$ is not included in \cite{CM} for technical reasons
on the structure of $\CD(K,\infty)$-spaces,
but the same argument gives \eqref{eq:BL} (corresponding to $N=D=\infty$)
for reversible Finsler manifolds.
The proof in \cite{CM} is based on the \emph{needle decomposition}
(also called \emph{localization}) inspired by Klartag's work \cite{Kl} on Riemannian manifolds,
extending the successful technique in convex geometry.
Along the lines of \cite{CM}, in \cite{Oneedle} we have generalized
the needle decomposition to non-reversible Finsler manifolds,
however, then we obtain only a weaker isoperimetric inequality,
\begin{equation}\label{eq:needle}
\cI_{(M,F,\fm)}(\theta) \ge \Lambda_F^{-1} \cdot \cI_{K,N,D}(\theta),
\end{equation}
with $\Lambda_F$ in \eqref{eq:revF}.
The inequality \eqref{eq:BL} improves \eqref{eq:needle}
in the case where $N=D=\infty$ and $K>0$,
and supports a conjecture that the sharp isoperimetric inequality
in the non-reversible case is the same as the reversible case,
namely $\Lambda_F^{-1}$ in \eqref{eq:needle} would be removed.

The organization of this article is as follows:
In Section~\ref{sc:prel} we review the basics of Finsler geometry,
including the weighted Ricci curvature and the Bochner--Weitzenb\"ock formula.
Section~\ref{sc:heat} is devoted to a detailed study of the nonlinear heat semigroup
and its linearizations, we improve the Bochner inequality under $\Ric_{\infty} \ge K$
and show the $L^1$-gradient estimate.
We prove the isoperimetric inequality in Section~\ref{sc:BL}.
\bigskip

{\it Acknowledgements.}
I am grateful to Kazumasa Kuwada for his suggestion to consider this problem
and for many valuable discussions.
I also thank Karl-Theodor Sturm and Kohei Suzuki for stimulating discussions.

\section{Geometry and analysis on Finsler manifolds}\label{sc:prel}%%%%%%%%%%%%%%%
%%%%%%%%%%%%%%%%

We review the basics of Finsler geometry (we refer to \cite{BCS,Shlec} for further reading),
and introduce the weighted Ricci curvature and the nonlinear Laplacian
studied in \cite{Oint,OShf} (see also \cite{GS} for the latter).

Throughout the article, let $M$ be a connected $\cC^{\infty}$-manifold
without boundary of dimension $n \ge 2$.
We also fix an arbitrary positive $\cC^{\infty}$-measure $\fm$ on $M$.

\subsection{Finsler manifolds}\label{ssc:Fmfd}%%%%%%%%%%%%%%%%%%%%%%
%%%%%%%%%%%%%%%%

Given local coordinates $(x^i)_{i=1}^n$ on an open set $U \subset M$,
we will always use the fiber-wise linear coordinates
$(x^i,v^j)_{i,j=1}^n$ of $TU$ such that
\[ v=\sum_{j=1}^n v^j \frac{\del}{\del x^j}\Big|_x \in T_xM, \qquad x \in U. \]

\begin{definition}[Finsler structures]\label{df:Fstr}
We say that a nonnegative function $F:TM \lra [0,\infty)$ is
a \emph{$\cC^{\infty}$-Finsler structure} of $M$ if the following three conditions hold:
\begin{enumerate}[(1)]
\item(\emph{Regularity})
$F$ is $\cC^{\infty}$ on $TM \setminus 0$,
where $0$ stands for the zero section;

\item(\emph{Positive $1$-homogeneity})
It holds $F(cv)=cF(v)$ for all $v \in TM$ and $c \ge 0$;

\item(\emph{Strong convexity})
The $n \times n$ matrix
\begin{equation}\label{eq:gij}
\big( g_{ij}(v) \big)_{i,j=1}^n :=
 \bigg( \frac{1}{2}\frac{\del^2 (F^2)}{\del v^i \del v^j}(v) \bigg)_{i,j=1}^n
\end{equation}
is positive-definite for all $v \in TM \setminus 0$.
\end{enumerate}
We call such a pair $(M,F)$ a \emph{$\cC^{\infty}$-Finsler manifold}.
\end{definition}

In other words, $F$ provides a Minkowski norm on each tangent space
which varies smoothly in horizontal directions.
If $F(-v)=F(v)$ holds for all $v \in TM$, then we say that $F$ is \emph{reversible}
or \emph{absolutely homogeneous}.
The strong convexity means that the unit sphere $T_xM \cap F^{-1}(1)$
(called the \emph{indicatrix}) is `positively curved' and implies the strict convexity:
$F(v+w) \le F(v)+F(w)$ for all $v,w \in T_xM$ and equality holds
only when $v=aw$ or $w=av$ for some $a \ge 0$.

In the coordinates $(x^i,\alpha_j)_{i,j=1}^n$ of $T^*U$ given by
$\alpha=\sum_{j=1}^n \alpha_j \,dx^j$,
we will also consider
\[ g^*_{ij}(\alpha) :=\frac{1}{2} \frac{\del^2[(F^*)^2]}{\del \alpha_i \del \alpha_j}(\alpha),
\qquad i,j=1,2,\ldots,n, \]
for $\alpha \in T^*U \setminus 0$.
Here $F^*:T^*M \lra [0,\infty)$ is the \emph{dual Minkowski norm} to $F$, namely
\[ F^*(\alpha) :=\sup_{v \in T_xM,\, F(v) \le 1} \alpha(v)
 =\sup_{v \in T_xM,\, F(v)=1} \alpha(v) \]
for $\alpha \in T_x^*M$.
It is clear by definition that $\alpha(v) \le F^*(\alpha) F(v)$, and hence
\[ \alpha(v) \ge -F^*(\alpha)F(-v), \qquad \alpha(v) \ge -F^*(-\alpha)F(v). \]
We remark that, however, $\alpha(v) \ge -F^*(\alpha)F(v)$ does not hold in general.

Let us denote by $\cL^*:T^*M \lra TM$ the \emph{Legendre transform}.
Precisely, $\cL^*$ is sending $\alpha \in T_x^*M$ to the unique element $v \in T_xM$
such that $F(v)=F^*(\alpha)$ and $\alpha(v)=F^*(\alpha)^2$.
In coordinates we can write down
\[ \cL^*(\alpha)=\sum_{i,j=1}^n g_{ij}^*(\alpha) \alpha_i \frac{\del}{\del x^j}\Big|_x
 =\sum_{j=1}^n \frac{1}{2} \frac{\del[(F^*)^2]}{\del \alpha_j}(\alpha) \frac{\del}{\del x^j}\Big|_x \]
for $\alpha \in T_x^*M \setminus 0$ (the latter expression makes sense also at $0$).
Note that $g^*_{ij}(\alpha) =g^{ij}(\cL^*(\alpha))$ for $\alpha \in T_x^*M \setminus 0$,
where $(g^{ij}(v))$ denotes the inverse matrix of $(g_{ij}(v))$.
The map $\cL^*|_{T^*_xM}$ is being a linear operator only when $F|_{T_xM}$
comes from an inner product.
We also define $\cL:=(\cL^*)^{-1}:TM \lra T^*M$.

For $x,y \in M$, we define the (asymmetric) \emph{distance} from $x$ to $y$ by
\[ d(x,y):=\inf_{\eta} \int_0^1 F\big( \dot{\eta}(t) \big) \,dt, \]
where $\eta:[0,1] \lra M$ runs over all $\cC^1$-curves such that $\eta(0)=x$ and $\eta(1)=y$.
Note that $d(y,x) \neq d(x,y)$ can happen since $F$ is only positively homogeneous.
A $\cC^{\infty}$-curve $\eta$ on $M$ is called a \emph{geodesic}
if it is locally minimizing and has a constant speed with respect to $d$,
similarly to Riemannian or metric geometry.
See \eqref{eq:geod} below for the precise geodesic equation.
For $v \in T_xM$, if there is a geodesic $\eta:[0,1] \lra M$
with $\dot{\eta}(0)=v$, then we define the \emph{exponential map}
by $\exp_x(v):=\eta(1)$.
We say that $(M,F)$ is \emph{forward complete} if the exponential
map is defined on whole $TM$.
Then the Hopf--Rinow theorem ensures that any pair of points
is connected by a minimal geodesic (see \cite[Theorem~6.6.1]{BCS}).

Given each $v \in T_xM \setminus 0$, the positive-definite matrix
$(g_{ij}(v))_{i,j=1}^n$ in \eqref{eq:gij} induces the Riemannian structure $g_v$ of $T_xM$ by
\begin{equation}\label{eq:gv}
g_v\bigg( \sum_{i=1}^n a_i \frac{\del}{\del x^i}\Big|_x,
 \sum_{j=1}^n b_j \frac{\del}{\del x^j}\Big|_x \bigg)
 := \sum_{i,j=1}^n g_{ij}(v) a_i b_j.
\end{equation}
Notice that this definition is coordinate-free and $g_v(v,v)=F^2(v)$ holds.
One can regard $g_v$ as the best Riemannian approximation of $F|_{T_xM}$
in the direction $v$.
The \emph{Cartan tensor}
\[ A_{ijk}(v):=\frac{F(v)}{2} \frac{\del g_{ij}}{\del v^k}(v),
 \qquad v \in TM \setminus 0, \]
measures the variation of $g_v$ in vertical directions,
and vanishes everywhere on $TM \setminus 0$
if and only if $F$ comes from a Riemannian metric.

The following useful fact on homogeneous functions (see \cite[Theorem~1.2.1]{BCS})
plays a fundamental role in our calculus.

\begin{theorem}[Euler's theorem]\label{th:Euler}
Suppose that a differentiable function $H:\R^n \setminus 0 \lra \R$ satisfies
$H(cv)=c^r H(v)$ for some $r \in \R$ and all $c>0$ and $v \in \R^n \setminus 0$
$($that is, \emph{positively $r$-homogeneous}$)$.
Then we have, for all $v \in \R^n \setminus 0$,
\[ \sum_{i=1}^n \frac{\del H}{\del v^i}(v)v^i=rH(v). \]
\end{theorem}

Observe that $g_{ij}$ is positively $0$-homogeneous on each $T_xM$, and hence
\begin{equation}\label{eq:Av}
\sum_{i=1}^n A_{ijk}(v)v^i =\sum_{j=1}^n A_{ijk}(v)v^j
 =\sum_{k=1}^n A_{ijk}(v)v^k =0
\end{equation}
for all $v \in TM \setminus 0$ and $i,j,k=1,2,\ldots,n$.
Define the \emph{formal Christoffel symbol}
\begin{equation}\label{eq:gamma}
\gamma^i_{jk}(v):=\frac{1}{2}\sum_{l=1}^n g^{il}(v) \bigg\{
 \frac{\del g_{lk}}{\del x^j}(v) +\frac{\del g_{jl}}{\del x^k}(v)
 -\frac{\del g_{jk}}{\del x^l}(v) \bigg\}
\end{equation}
for $v \in TM \setminus 0$,  and the \emph{geodesic spray coefficients}
and the \emph{nonlinear connection}
\[ G^i(v):=\sum_{j,k=1}^n \gamma^i_{jk}(v) v^j v^k, \qquad
 N^i_j(v):=\frac{1}{2} \frac{\del G^i}{\del v^j}(v) \]
for $v \in TM \setminus 0$ ($G^i(0)=N^i_j(0):=0$ by convention).
Note that $G^i$ is positively $2$-homogeneous, hence
Theorem~\ref{th:Euler} implies $\sum_{j=1}^n N^i_j(v) v^j=G^i(v)$.

By using $N^i_j$,
the coefficients of the \emph{Chern connection} are given by
\begin{equation}\label{eq:Gamma}
\Gamma^i_{jk}(v):=\gamma^i_{jk}(v)
 -\sum_{l,m=1}^n \frac{g^{il}}{F}(A_{lkm}N^m_j +A_{jlm}N^m_k -A_{jkm}N^m_l)(v)
\end{equation}
on $TM \setminus 0$.
The corresponding \emph{covariant derivative} of a vector field $X$ by $v \in T_xM$
with \emph{reference vector} $w \in T_xM \setminus 0$ is defined as
\begin{equation}\label{eq:covd}
D_v^w X(x):=\sum_{i,j=1}^n \bigg\{ v^j \frac{\del X^i}{\del x^j}(x)
 +\sum_{k=1}^n \Gamma^i_{jk}(w) v^j X^k(x) \bigg\} \frac{\del}{\del x^i}\Big|_x \in T_xM.
\end{equation}
Then the \emph{geodesic equation} is written as, with the help of \eqref{eq:Av},
\begin{equation}\label{eq:geod}
D_{\dot{\eta}}^{\dot{\eta}} \dot{\eta}(t)
 =\sum_{i=1}^n \big\{ \ddot{\eta}^i(t) +G^i\big( \dot{\eta}(t) \big) \big\}
 \frac{\del}{\del x^i} \Big|_{\eta(t)} =0.
\end{equation}

\subsection{Uniform convexity and smoothness}\label{ssc:us}%%%%%%%%%%%%%%%%%%
%%%%%%%%%%%%%%%%

We will need the following quantity associated with $(M,F)$:
\[ \sS_F :=\sup_{x \in M} \sup_{v,w \in T_xM \setminus 0} \frac{g_v(w,w)}{F^2(w)} \,\in [1,\infty]. \]
Since $g_v(w,w) \le \sS_F F^2(w)$ and $g_v$ is the Hessian of $F^2/2$ at $v$,
the constant $\sS_F$ measures the (fiber-wise) concavity of $F^2$
and is called the ($2$-)\emph{uniform smoothness constant} (see \cite{Ouni}).
We remark that $\sS_F=1$ holds if and only if $(M,F)$ is Riemannian.
The following lemma is a standard fact, we give a proof for thoroughness.

\begin{lemma}\label{lm:us}
For any $x \in M$, $v \in T_xM \setminus 0$ and $\alpha:=\cL(v)$, we have
\[ \sup_{w \in T_xM \setminus 0} \frac{g_v(w,w)}{F^2(w)}
 =\sup_{\beta \in T^*_xM \setminus 0}
 \frac{F^*(\beta)^2}{g^*_{\alpha}(\beta,\beta)}, \]
where $g^*_{\alpha}$ is the inner product of $T^*_xM$ defined by
\[ g^*_{\alpha}(\beta,\beta):=\sum_{i,j=1}^n g^*_{ij}(\alpha) \beta_i \beta_j,
 \qquad \beta=\sum_{i=1}^n \beta_i \,dx^i. \]
\end{lemma}

\begin{proof}
Choose local coordinates $(x^i)_{i=1}^n$ around $x$ such that
$g_{ij}(v)=\delta_{ij}$ and set
\begin{align*}
\Sph_x &:=\bigg\{ w=\sum_{i=1}^n w^i \frac{\del}{\del x^i} \in T_xM
 \,\bigg|\, \sum_{i=1}^n (w^i)^2=1 \bigg\}, \\
\Sph_x^* &:=\bigg\{ \beta=\sum_{i=1}^n \beta_i \,dx^i \in T_x^*M
 \,\bigg|\, \sum_{i=1}^n (\beta_i)^2=1 \bigg\}.
\end{align*}
First, given $w \in \Sph_x$, we take $\beta \in \Sph_x^*$ such that $\beta(w)=1$.
Then we have $1=\beta(w) \le F^*(\beta) F(w)$ and hence
\[ \frac{g_v(w,w)}{F^2(w)} =\frac{1}{F^2(w)} \le F^*(\beta)^2
 =\frac{F^*(\beta)^2}{g^*_{\alpha}(\beta,\beta)}. \]
Next, for $\beta' \in \Sph_x^*$, take $w' \in \Sph_x$ with $\beta'(w')=F^*(\beta')F(w')$.
Then $F^*(\beta')F(w')=\beta'(w') \le 1$ and hence $1/F^2(w') \ge F^*(\beta')^2$.
This completes the proof.
$\qedd$
\end{proof}

One can in a similar manner introduce the ($2$-)\emph{uniform convexity constant}:
\begin{equation}\label{eq:uc}
\sC_F :=\sup_{x \in M} \sup_{v,w \in T_xM \setminus 0} \frac{F^2(w)}{g_v(w,w)}
 =\sup_{x \in M} \sup_{\alpha,\beta \in T^*_xM \setminus 0}
 \frac{g^*_{\alpha}(\beta,\beta)}{F^*(\beta)^2} \,\in [1,\infty].
\end{equation}
Again, $\sC_F=1$ holds if and only if $(M,F)$ is Riemannian.
We remark that $\sS_F$ and $\sC_F$ control the \emph{reversibility constant} $\Lambda_F$
defined in \eqref{eq:revF} as follows.

\begin{lemma}\label{lm:rev}
We have
\[ \Lambda_F \le \min\{ \sqrt{\sS_F},\sqrt{\sC_F} \}. \]
\end{lemma}

\begin{proof}
For any $v \in TM \setminus 0$, we observe
\[ \frac{F^2(v)}{F^2(-v)} =\frac{g_v(v,v)}{F^2(-v)}
 =\frac{g_v(-v,-v)}{F^2(-v)} \le \sS_F, \]
and similarly
\[ \frac{F^2(v)}{F^2(-v)} =\frac{F^2(v)}{g_{-v}(v,v)} \le \sC_F. \]
$\qedd$
\end{proof}

\subsection{Weighted Ricci curvature}\label{ssc:wRic}%%%%%%%%%%%%%%%%%
%%%%%%%%%%%%%%%%

The \emph{Ricci curvature} (as the trace of the \emph{flag curvature}) on a Finsler manifold
is defined by using some connection.
Instead of giving a precise definition in coordinates (for which we refer to \cite{BCS}),
here we explain a useful interpretation in \cite[\S 6.2]{Shlec} going back to (at least) \cite{Au}.
Given a unit vector $v \in T_xM \cap F^{-1}(1)$, we extend it to a $\cC^{\infty}$-vector field $V$
on a neighborhood $U$ of $x$ in such a way that every integral curve of $V$ is geodesic,
and consider the Riemannian structure $g_V$ of $U$ induced from \eqref{eq:gv}.
Then the \emph{Finsler} Ricci curvature $\Ric(v)$ of $v$ with respect to $F$ coincides with
the \emph{Riemannian} Ricci curvature of $v$ with respect to $g_V$
(in particular, it is independent of the choice of $V$).

Inspired by the above interpretation of the Ricci curvature
as well as the theory of weighted Ricci curvature
(also called the \emph{Bakry--\'Emery--Ricci curvature}) of Riemannian manifolds,
the \emph{weighted Ricci curvature} for $(M,F,\fm)$ was introduced in \cite{Oint} as follows.
Recall that $\fm$ is a positive $\cC^{\infty}$-measure on $M$,
from here on it comes into play.

\begin{definition}[Weighted Ricci curvature]\label{df:wRic}
Given a unit vector $v \in T_xM$, let $V$ be a $\cC^{\infty}$-vector field
on a neighborhood $U$ of $x$ as above.
We decompose $\fm$ as $\fm=\e^{-\Psi}\vol_{g_V}$ on $U$,
where $\Psi \in \cC^{\infty}(U)$ and $\vol_{g_V}$ is the volume form of $g_V$.
Denote by $\eta:(-\ve,\ve) \lra M$ the geodesic such that $\dot{\eta}(0)=v$.
Then, for $N \in (-\infty,0) \cup (n,\infty)$, define
\[ \Ric_N(v):=\Ric(v) +(\Psi \circ \eta)''(0) -\frac{(\Psi \circ \eta)'(0)^2}{N-n}. \]
We also define as the limits:
\[ \Ric_{\infty}(v):=\Ric(v) +(\Psi \circ \eta)''(0), \qquad
 \Ric_n(v):=\lim_{N \downarrow n}\Ric_N(v). \]
For $c \ge 0$, we set $\Ric_N(cv):=c^2 \Ric_N(v)$.
\end{definition}

We will denote by $\Ric_N \ge K$, $K \in \R$, the condition
$\Ric_N(v) \ge KF^2(v)$ for all $v \in TM$.
In the Riemannian case, the study of $\Ric_{\infty}$ goes back to Lichnerowicz \cite{Li},
he showed a Cheeger--Gromoll type splitting theorem (see \cite{Osplit} for a Finsler counterpart).
The range $N \in (n,\infty)$ has been well studied by Bakry \cite[\S 6]{Ba},
Qian \cite{Qi} and many others.
The study of the range $N \in (-\infty,0)$ is more recent;
see \cite{Mineg} for isoperimetric inequalities, \cite{Oneg} for the curvature-dimension condition,
and \cite{Wy} for splitting theorems (for $N \in (-\infty,1]$).

It was established in \cite{Oint} (and \cite{Oneg} for $N<0$, \cite{Oneedle} for $N=0$) that,
for $K \in \R$, the bound $\Ric_N \ge K$  is equivalent to Lott, Sturm and Villani's
\emph{curvature-dimension condition} $\CD(K,N)$.
This extends the corresponding result on weighted Riemannian manifolds
and has many geometric and analytic applications (see \cite{Oint,OShf} among others).

\begin{remark}[$\bS$-curvature]\label{rm:S-curv}
For a Riemannian manifold $(M,g,\vol_g)$ endowed with the Riemannian volume measure,
clearly we have $\Psi \equiv 0$ and hence $\Ric_N =\Ric$ for all $N$.
It is also known that, for Finsler manifolds of \emph{Berwald type}
(i.e., $\Gamma_{ij}^k$ is constant on each $T_xM \setminus 0$),
the \emph{Busemann--Hausdorff measure} satisfies $(\Psi \circ \eta)' \equiv 0$
(in other words, Shen's \emph{$\bS$-curvature} vanishes, see \cite[\S 7.3]{Shlec}).
For a general Finsler manifold, however, there may not exist any measure with vanishing
$\bS$-curvature (see \cite{ORand} for such an example).
This is a reason why we chose to begin with an arbitrary measure $\fm$.
\end{remark}

For later convenience, we introduce the following notations.

\begin{definition}[Reverse Finsler structures]\label{df:rev}
We define the \emph{reverse Finsler structure} $\rev{F}$ of $F$ by
$\rev{F}(v):=F(-v)$.
\end{definition}

We will put an arrow $\leftarrow$ on those quantities associated with $\rev{F}$,
we have for example $\rev{d}\!(x,y)=d(y,x)$, $\rev{\Ric}_N(v)=\Ric_N(-v)$
and $\rev{\Grad}u=-\Grad(-u)$.
We say that $(M,F)$ is \emph{backward complete} if $(M,\rev{F})$ is forward complete.
If $\Lambda_F<\infty$, then these completenesses are mutually equivalent,
and we may call it simply \emph{completeness}.

\subsection{Nonlinear Laplacian and heat flow}\label{ssc:heat}%%%%%%%%%%%%%%
%%%%%%%%%%%%%%%%

For a differentiable function $u:M \lra \R$, the \emph{gradient vector}
at $x$ is defined as the Legendre transform of the derivative of $u$:
$\Grad u(x):=\cL^*(du(x)) \in T_xM$.
If $du(x) \neq 0$, then we can write down in coordinates as
\[ \Grad u
 =\sum_{i,j=1}^n g^*_{ij}(du) \frac{\del u}{\del x^j} \frac{\del}{\del x^i}. \]
We need to be careful when $du(x)=0$,
because $g^*_{ij}(du(x))$ is not defined as well as
the Legendre transform $\cL^*$ is only continuous at the zero section.
Therefore we set
\[ M_u:=\{ x \in M \,|\, du(x) \neq 0 \}. \]
For a twice differentiable function $u:M \lra \R$ and $x \in M_u$,
we define a kind of \emph{Hessian} $\Grad^2 u(x) \in T_x^*M \otimes T_xM$
by using the covariant derivative \eqref{eq:covd} as
\[ \Grad^2 u(v):=D^{\Grad u}_v (\Grad u)(x) \,\in T_xM, \qquad v \in T_xM. \]
The operator $\Grad^2 u(x)$ is symmetric in the sense that
\[ g_{\Grad u}\big( \Grad^2 u(v),w \big)=g_{\Grad u}\big( v,\Grad^2 u(w) \big) \]
for all $v,w \in T_xM$ with $x \in M_u$
(see, for example, \cite[Lemma~2.3]{OSbw}).

Define the \emph{divergence} of a differentiable vector field $V$ on $M$
with respect to the measure $\fm$ by
\[ \div_{\fm} V:=\sum_{i=1}^n \bigg( \frac{\del V^i}{\del x^i} +V^i \frac{\del \Phi}{\del x^i} \bigg),
 \qquad V=\sum_{i=1}^n V^i \frac{\del}{\del x^i}, \]
where we decomposed $\fm$ as $d\fm=\e^{\Phi} \,dx^1 dx^2 \cdots dx^n$.
One can rewrite in the weak form as
\[ \int_M \phi \div_{\fm} V \,d\fm =-\int_M d\phi(V) \,d\fm \qquad
 \text{for all}\ \phi \in \cC_c^{\infty}(M), \]
that makes sense for measurable vector fields $V$ with $F(V) \in L_{\loc}^1(M)$.
Then we define the distributional \emph{Laplacian} of $u \in H^1_{\loc}(M)$ by
$\Lap u:=\div_{\fm}(\Grad u)$ in the weak sense that
\[ \int_M \phi\Lap u \,d\fm:=-\int_M d\phi(\Grad u) \,d\fm \qquad
 \text{for all}\ \phi \in \cC_c^{\infty}(M). \]
Notice that the space $H^1_{\loc}(M)$ is defined solely in terms of the differentiable structure of $M$.
Since taking the gradient vector (more precisely, the Legendre transform)
is a nonlinear operation, our Laplacian $\Lap$ is a nonlinear operator
unless $F$ is Riemannian.

In \cite{OShf,OSbw}, we have studied the associated \emph{nonlinear heat equation}
$\del_t u=\Lap u$.
In order to recall some results in \cite{OShf},
we define the \emph{Dirichlet energy} of $u \in H_{\loc}^1(M)$ by
\[ \cE(u):=\frac{1}{2}\int_M F^2(\Grad u) \,d\fm
 =\frac{1}{2}\int_M F^*(du)^2 \,d\fm. \]
We remark that $\cE(u)<\infty$ does not necessarily imply $\cE(-u)<\infty$.
Define $H^1_0(M)$ as the closure of $\cC_c^{\infty}(M)$
with respect to the (absolutely homogeneous) norm
\[ \|u\|_{H^1}:=\|u\|_{L^2} +\{ \cE(u)+\cE(-u) \}^{1/2}. \]
Note that $(H^1_0(M),\|\cdot\|_{H^1})$ is a Banach space.

\begin{definition}[Global solutions]\label{df:hf}
We say that a function $u$ on $[0,T] \times M$, $T>0$,
is a \emph{global solution} to the heat equation $\del_t u=\Lap u$ if
it satisfies the following$:$
\begin{enumerate}[(1)]
\item $u \in L^2\big( [0,T],H^1_0(M) \big) \cap H^1 \big( [0,T],H^{-1}(M) \big)$;

\item For every $\phi \in \cC_c^{\infty}(M)$, we have
\[ \int_M \phi \cdot \del_t u_t \,d\fm =-\int_M d\phi(\Grad u_t) \,d\fm \]
for almost all $t \in [0,T]$, where we set $u_t:=u(t,\cdot)$.
\end{enumerate}
\end{definition}

We refer to \cite{Ev} for the notations as in (1).
Denoted by $H^{-1}(M)$ is the dual Banach space of $H^1_0(M)$
(so that $H^1_0(M) \subset L^2(M) \subset H^{-1}(M)$).
By noticing
\begin{align*}
\int_M |(d\phi-d\bar{\phi})(\Grad u_t)| \,d\fm
&\le \int_M \max\big\{ F^* \big( d(\phi-\bar{\phi}) \big),F^* \big( d(\bar{\phi}-\phi) \big) \big\}
 F(\Grad u_t) \,d\fm \\
&\le \{ 2\cE(\phi-\bar{\phi}) +2\cE(\bar{\phi}-\phi) \}^{1/2} \cdot \{ 2\cE(u_t) \}^{1/2},
\end{align*}
the test function $\phi$ can be taken from $H^1_0(M)$.
Global solutions can be constructed as gradient curves of the energy functional $\cE$
in the Hilbert space $L^2(M)$.
We summarize the existence and regularity properties established in \cite[\S\S 3, 4]{OShf}
in the next theorem.

\begin{theorem}\label{th:hf}
Assume $\Lambda_F<\infty$.

\begin{enumerate}[{\rm (i)}]
\item
For each initial datum $u_0 \in H^1_0(M)$ and $T>0$,
there exists a unique global solution $u$ to the heat equation on $[0,T] \times M$,
and the distributional Laplacian $\Lap u_t$ is absolutely continuous
with respect to $\fm$ for all $t \in (0,T)$.

\item
One can take the continuous version of a global solution $u$, and it enjoys
the $H^2_{\loc}$-regularity in $x$ as well as the $\cC^{1,\alpha}$-regularity for some $\alpha$
in both $t$ and $x$.
Moreover, $\del_t u$ lies in $H^1_{\loc}(M) \cap \cC(M)$,
and further in $H_0^1(M)$ if $\sS_F <\infty$.
\end{enumerate}
\end{theorem}

We remark that the usual elliptic regularity yields that
$u$ is $\cC^{\infty}$ on $\bigcup_{t>0}(\{t\} \times M_{u_t})$.
The proof of $\del_t u \in H_0^1(M)$ under $\sS_F <\infty$
can be found in \cite[Appendix~A]{OShf}.
The uniqueness in (i) is a consequence of the convexity of $F^*$ (see \cite[Proposition~3.5]{OShf}).

We finally remark that, by the construction of heat flow as the gradient flow of $\cE$,
it is readily seen that:
\begin{equation}\label{eq:u>=0}
\text{If}\ u_0 \ge 0\ \text{almost everywhere, then}\ u_t \ge 0\ \text{almost everywhere for all}\ t>0.
\end{equation}
Indeed, if $u_t<0$ on a non-null set,
then the curve $\bar{u}_t:=\max\{ u_t,0 \}$ will give a less energy
with a less $L^2$-length, a contradiction.

\subsection{Bochner--Weitzenb\"ock formula}\label{ssc:BW}%%%%%%%%%%%%%%%%
%%%%%%%%%%%%%%%%

Given $f \in H^1_{\loc}(M)$ and a measurable vector field $V$
such that $V \neq 0$ almost everywhere on $M_f=\{ x \in M \,|\, df(x) \neq 0 \}$,
we can define the gradient vector field and the Laplacian
on the weighted Riemannian manifold $(M,g_V,\fm)$ by
\[ \nabla^V f:=\left\{ \begin{array}{ll}
 \displaystyle\sum_{i,j=1}^n g^{ij}(V) \frac{\del f}{\del x^j} \frac{\del}{\del x^i}
 & \text{on}\ M_f, \smallskip\\
 0 & \text{on}\ M \setminus M_f, \end{array}\right.
 \qquad \Delta\!^V f:=\div_{\fm}(\nabla^V f), \]
where the latter is in the sense of distribution.
We have $\nabla^{\Grad u}u=\Grad u$ and $\Delta\!^{\Grad u}u=\Lap u$
for $u \in H^1_{\loc}(M)$ (\cite[Lemma~2.4]{OShf}).
We also observe that, for $f_1,f_2 \in H^1_{\loc}(M)$
and $V$ such that $V \neq 0$ almost everywhere,
\begin{equation}\label{eq:f1f2}
df_2(\nabla^V f_1) =g_V(\nabla^V f_1,\nabla^V f_2) =df_1(\nabla^V f_2).
\end{equation}

We established in \cite[Theorem~3.3]{OSbw} the following key ingredient of the $\Gamma$-calculus.

\begin{theorem}[Bochner--Weitzenb\"ock formula]\label{th:BW}
Given $u \in \cC^{\infty}(M)$, we have
\begin{equation}\label{eq:BW}
\Delta\!^{\Grad u} \bigg[ \frac{F^2(\Grad u)}{2} \bigg] -d(\Lap u)(\Grad u)
 =\Ric_{\infty}(\Grad u) +\| \Grad^2 u \|_{\HS(\Grad u)}^2
\end{equation}
as well as
\[ \Delta\!^{\Grad u} \bigg[ \frac{F^2(\Grad u)}{2} \bigg] -d(\Lap u)(\Grad u)
 \ge \Ric_N(\Grad u) +\frac{(\Lap u)^2}{N} \]
for $N \in (-\infty,0) \cup [n,\infty]$ point-wise on $M_u$,
where $\|\cdot\|_{\HS(\Grad u)}$ denotes the Hilbert--Schmidt norm with respect to $g_{\Grad u}$.
\end{theorem}

In particular, if $\Ric_N \ge K$, then we have
\begin{equation}\label{eq:Boc}
\Delta\!^{\Grad u} \bigg[ \frac{F^2(\Grad u)}{2} \bigg] -d(\Lap u)(\Grad u)
 \ge KF^2(\Grad u) +\frac{(\Lap u)^2}{N}
\end{equation}
on $M_u$, that we will call the \emph{Bochner inequality}.
One can further generalize the Bochner--Weitzenb\"ock formula
to a more general class of Hamiltonian systems
(by dropping the positive $1$-homogeneity; see \cite{Lee,Oham}).

\begin{remark}[$F$ versus $g_{\Grad u}$]\label{rm:g_u}
In contrast to $\Delta\!^{\Grad u}u=\Lap u$, $\Ric_N(\Grad u)$
may not coincide with the weighted Ricci curvature $\Ric_N^{\Grad u}(\Grad u)$
of the weighted Riemannian manifold $(M,g_{\Grad u},\fm)$.
It is compensated in \eqref{eq:BW} by the fact that $\Grad^2 u$ does not
necessarily coincide with the Hessian of $u$ with respect to $g_{\Grad u}$.
\end{remark}

The integrated form was shown in \cite[Theorem~3.6]{OSbw},
with the help of the following fact to overcome the ill-posedness of $\Grad u$ on $M \setminus M_u$
(see \cite[Exercise~10.37(iv)]{Leo}, \cite[Lemma~1.7.1]{Maz} for example).

\begin{lemma}\label{lm:BWint}
For each $f \in H^1_{\loc}(M)$, we have $df=0$ almost everywhere on $f^{-1}(0)$.
If $f \in H^1_{\loc}(M) \cap L^{\infty}_{\loc}(M)$,
then $d(f^2/2)=f \,df=0$ also holds almost everywhere on $f^{-1}(0)$.
\end{lemma}

\begin{theorem}[Integrated form]\label{th:BWint}
Assume $\Ric_N \ge K$ for some $K \in \R$ and $N \in (-\infty,0) \cup [n,\infty]$.
Given $u \in H^2_{\loc}(M) \cap \cC^1(M)$ such that $\Lap u \in H_{\loc}^1(M)$,
we have
%\begin{equation}\label{eq:BWint}
\[ -\int_M d\phi \bigg( \nabla^{\Grad u} \bigg[ \frac{F^2(\Grad u)}{2} \bigg] \bigg) \,d\fm
 \ge \int_M \phi \bigg\{ d(\Lap u)(\Grad u) +K F^2(\Grad u) +\frac{(\Lap u)^2}{N} \bigg\} \,d\fm \]
%\end{equation}
for all bounded nonnegative functions $\phi \in H_c^1(M) \cap L^{\infty}(M)$.
\end{theorem}

Recall from Theorem~\ref{th:hf}(ii) that global solutions to the heat equation always enjoy
$u \in H_0^1(M) \cap H^2_{\loc}(M) \cap \cC^1(M)$ and $\Lap u \in H^1_{\loc}(M)$.

\section{Linearized semigroups and gradient estimates}\label{sc:heat}%%%%%%%%%
%%%%%%%%%%%%%%%%

In the Bochner--Weitzenb\"ock formula (Theorem~\ref{th:BW}) in the previous section,
we used the linearized Laplacian $\Delta\!^{\Grad u}$ induced from
the Riemannian structure $g_{\Grad u}$.
In the same spirit, we can consider the linearized heat equation
associated with a global solution to the heat equation.
This technique turned out useful and we have obtained gradient estimates
\`a la Bakry--\'Emery and Li--Yau in \cite[\S 4]{OSbw}.
In this section we discuss such a linearization in detail
and improve the $L^2$-gradient estimate to an $L^1$-bound (Theorem~\ref{th:L1}).

\subsection{Linearized heat semigroups and their adjoints}\label{ssc:lin}%%%%%%%%%
%%%%%%%%%%%%%%%%%%%%%

Let $(u_t)_{t \ge 0}$ be a global solution to the heat equation.
We will fix a measurable one-parameter family of \emph{non-vanishing} vector fields
$(V_t)_{t \ge 0}$ such that $V_t=\Grad u_t$ on $M_{u_t}$ for each $t \ge 0$.
Given $f \in H^1_0(M)$ and $s \ge 0$,
let $(P_{s,t}^{\Grad u}(f))_{t \ge s}$ be the weak solution to the \emph{linearized heat equation}:
\begin{equation}\label{eq:lin-hf}
\del_t [P_{s,t}^{\Grad u}(f)] =\Delta\!^{V_t}[P_{s,t}^{\Grad u}(f)],
 \qquad P_{s,s}^{\Grad u}(f)=f.
\end{equation}
The existence and other properties of the linearized semigroup $P_{s,t}^{\Grad u}$
are summarized in the following proposition.

\begin{proposition}[Properties of linearized semigroups]\label{pr:lin}
Assume that $(M,F,\fm)$ is complete and satisfies $\sC_F<\infty$ and $\sS_F<\infty$,
and let $(u_t)_{t \ge 0}$ and $(V_t)_{t \ge 0}$ be as above.
\begin{enumerate}[{\rm (i)}]
\item
For each $s \ge 0$, $T>0$ and $f \in H^1_0(M)$, there exists a unique weak solution
$f_t=P^{\Grad u}_{s,t}(f)$, $t \in [s,s+T]$, to \eqref{eq:lin-hf}.
Moreover, $(f_t)_{t \in [s,s+T]}$ lies in $L^2([s,s+T],H^1_0(M)) \cap H^1([s,s+T],H^{-1}(M))$
as well as $\cC([s,s+T],L^2(M))$.
%in the weak sense that
%\begin{equation}\label{eq:lin-whf}
%\int_s^{s+T} \int_M \del_t \phi_t \cdot f_t \,d\fm \,dt
% =\int_s^{s+T} \int_M d\phi_t (\nabla^{V_t} f_t) \,d\fm \,dt
%\end{equation}
%for all $\phi \in \cC_c^{\infty}((s,s+T) \times M)$.

\item
The solution $(f_t)_{t \in [s,s+T]}$ in {\rm (i)} is H\"older continuous on $(s,s+T) \times M$.

\item
Assume that either $\fm(M)<\infty$ or $\Ric_{\infty} \ge K$ for some $K \in \R$ holds.
If $c \le f \le C$ for some $-\infty<c<C<\infty$,
then we have $c \le f_t \le C$ almost everywhere for all $t \in (s,s+T]$.
\end{enumerate}
\end{proposition}

\begin{proof}
(i)
Let $s=0$ without loss of generality.
This unique existence follows from Theorem~4.1 and Remark~4.3 in \cite[Chapter~III]{LM}
(see also \cite[Theorem~11.3]{RR},
where $A(t)$ is assumed to be continuous in $t$ but it is in fact unnecessary).
Precisely, in the notations in \cite{LM},
we take $H=L^2(M)$, $V=H^1_0(M)$, and put $A_t:=-\Delta\!^{V_t}:H^1_0(M) \lra H^{-1}(M)$.
We deduce with the help of \eqref{eq:uc} that, for any $h,\bar{h} \in H^1_0(M)$,
\[ \bigg| \int_M \bar{h} \Delta\!^{V_t} h \,d\fm \bigg|
 = \bigg| \int_M g^*_{\cL(V_t)}(dh,d\bar{h}) \,d\fm \bigg|
 \le 2 \sqrt{\cE^{V_t}(h)} \sqrt{\cE^{V_t}(\bar{h})}
 \le 2\sC_F \sqrt{\cE(h)} \sqrt{\cE(\bar{h})} \]
and
\[ -\int_M h \Delta\!^{V_t} h \,d\fm =2\cE^{V_t}(h)
 \ge \frac{2}{\sS_F} \cE(h), \]
where $\cE^{V_t}(h):=(1/2) \int_M g^*_{\cL(V_t)}(dh,dh) \,d\fm$ denotes the energy functional
on $(M,g_{V_t},\fm)$.
Since $\Lambda_F<\infty$ by $\sC_F<\infty$ (or $\sS_F<\infty$),
$\|h\|_{L^2}+\sqrt{\cE(h)}$ is comparable with $\|h\|_{H^1}$.
Therefore we have a unique solution $(f_t)_{t \in [0,T]}$ to \eqref{eq:lin-hf} with $f_0=f$
lying in $L^2([0,T],H^1_0(M)) \cap H^1([0,T],H^{-1}(M))$, and also in $\cC([0,T],L^2(M))$
(see \cite[\S 5.9.2]{Ev}, \cite[Lemma~11.4]{RR}).

(ii)
The H\"older continuity is a consequence of the local uniform ellipticity of $\Delta\!^{V_t}$
(see \cite[Proposition~4.4]{OShf}).

(iii)
This is seen for example by using the fundamental solution $q(t,x;s,y)$
to the equation $\del_t [P_{s,t}^{\Grad u}(f)] =\Delta\!^{V_t}[P_{s,t}^{\Grad u}(f)]$
(see \cite[\S 6]{Sal}).
We have
\[ f_t(x) =\int_M q(t,x;s,y) f(y) \,\fm(dy), \]
and $\int_M q(t,x;s,y) \,\fm(dy)=1$ (by $1 \in H^1_0(M)$ when $\fm(M)<\infty$,
or by \cite[\S 7]{Sal} since $\Ric_{\infty} \ge K$ implies the squared exponential volume bound
as in \cite[Theorem~4.24]{StI}).
This completes the proof.
$\qedd$
\end{proof}

The uniqueness in (i) above ensures that $u_t=P_{s,t}^{\Grad u}(u_s)$.
It follows from the non-expansion property,
\[ \frac{d}{dt}\big[ \|f_t\|_{L^2}^2 \big] =-4\cE^{V_t}(f_t) \le 0, \]
that $P^{\Grad u}_{s,t}$ uniquely extends to a linear contraction semigroup acting on $L^2(M)$.
Notice also that $f$ is $\cC^{\infty}$ on $\bigcup_{s<t<s+T} (\{t\} \times M_{u_t})$.

The operator $P_{s,t}^{\Grad u}$ is linear but not symmetric
(with respect to the $L^2$-inner product).
Let us denote by $\widehat{P}^{\Grad u}_{s,t}$ the \emph{adjoint operator}
of $P_{s,t}^{\Grad u}$.
That is to say, given $\phi \in H_0^1(M)$ and $t>0$,
we define $(\widehat{P}^{\Grad u}_{s,t}(\phi))_{s \in [0,t]}$ as the solution
to the equation
\begin{equation}\label{eq:Phat}
\del_s [\widehat{P}^{\Grad u}_{s,t}(\phi)]
 =-\Delta\!^{V_s} [\widehat{P}^{\Grad u}_{s,t}(\phi)],
 \qquad \widehat{P}^{\Grad u}_{t,t}(\phi)=\phi.
\end{equation}
Note that
\begin{equation}\label{eq:adj}
\int_M \phi \cdot P^{\Grad u}_{s,t}(f) \,d\fm
 =\int_M \widehat{P}^{\Grad u}_{s,t}(\phi) \cdot f \,d\fm
\end{equation}
indeed holds, since for $r \in (0,t-s)$
\begin{align*}
&\del_r \bigg[ \int_M \widehat{P}^{\Grad u}_{s+r,t}(\phi) \cdot
 P_{s,s+r}^{\Grad u}(f) \,d\fm \bigg] \\
&= -\int_M \Delta\!^{V_{s+r}}[\widehat{P}^{\Grad u}_{s+r,t}(\phi)] \cdot
 P_{s,s+r}^{\Grad u}(f) \,d\fm
 +\int_M \widehat{P}^{\Grad u}_{s+r,t}(\phi) \cdot
 \Delta\!^{V_{s+r}}[P_{s,s+r}^{\Grad u}(f)] \,d\fm \\
&=0.
\end{align*}
One may rewrite \eqref{eq:Phat} as
\[ \del_{\sigma} [\widehat{P}^{\Grad u}_{t-\sigma,t}(\phi)]
 =\Delta\!^{V_{t-\sigma}} [\widehat{P}^{\Grad u}_{t-\sigma,t}(\phi)],
 \qquad \sigma \in [0,t], \]
to see that the adjoint heat semigroup solves the linearized heat equation \emph{backward in time}.
(This evolution is sometimes called the \emph{conjugate heat semigroup},
especially in the Ricci flow theory; see for instance \cite[Chapter~5]{Ch+}.)
Therefore we see in the same way as $P^{\Grad u}_{s,t}$ that
$\|\widehat{P}^{\Grad u}_{t-\sigma,t}(\phi)\|_{L^2}$ is non-increasing in $\sigma$ and that
$\widehat{P}^{\Grad u}_{t-\sigma,t}$ extends to a linear contraction semigroup acting on $L^2(M)$.

\begin{remark}\label{rm:V_t}
In general, the semigroups $P^{\Grad u}_{s,t}$ and $\widehat{P}^{\Grad u}_{s,t}$
depend on the choice of an auxiliary vector field $(V_t)_{t \ge 0}$.
We will not discuss this issue, but carefully replace $V_t$ with $\Grad u_t$
as far as it is possible (with the help of Lemma~\ref{lm:BWint}).
\end{remark}

By a well known technique based on the Bochner inequality \eqref{eq:Boc} with $N=\infty$,
we obtained in \cite[Theorem~4.1]{OSbw} the \emph{$L^2$-gradient estimate} of the following form.

\begin{theorem}[$L^2$-gradient estimate, compact case]\label{th:L2}
Assume that $(M,F,\fm)$ is compact and satisfies $\Ric_{\infty} \ge K$ for some $K \in \R$.
Then, given any global solution $(u_t)_{t \ge 0}$ to the heat equation, we have
\[ F^2\big( \Grad u_t(x) \big)
 \le \e^{-2K(t-s)} P_{s,t}^{\Grad u} \big( F^2(\Grad u_s) \big) (x) \]
for all $0<s<t<\infty$ and $x \in M$.
\end{theorem}

We remark that, by Theorem~\ref{th:hf}, $F^2(\Grad u_s) \in H^1(M)$
and both sides in Theorem~\ref{th:L2} are H\"older continuous.
Let us stress that we use the nonlinear semigroup ($u_s \to u_t$) in the LHS,
while in the RHS the linearized semigroup $P^{\Grad u}_{s,t}$ is employed.

\begin{remark}\label{rm:OSbw}
In the proof of \cite[Theorem~4.1]{OSbw},
we did not distinguish $P^{\Grad u}_{s,t}$ and $\widehat{P}^{\Grad u}_{s,t}$
and treated $P^{\Grad u}_{s,t}$ as a symmetric operator.
However, the proof is valid by replacing $P^{\Grad u}_{s,t}(h)$ with
$\widehat{P}^{\Grad u}_{s,t}(h)$.
See the proof of Theorem~\ref{th:L1} below
which is based on a similar calculation (with the sharper inequality in Proposition~\ref{pr:Boc+}).
\end{remark}

\subsection{Improved Bochner inequality}\label{ssc:Boc+}%%%%%%%%%
%%%%%%%%%%%%%%%%%%%%%%

We shall give an inequality improving the Bochner inequality \eqref{eq:Boc} with $N=\infty$,
that will be used to show the $L^1$-gradient estimate as well as the isoperimetric inequality.
In the context of linear diffusion operators,
such an inequality can be derived from \eqref{eq:Boc} by a self-improvement argument
(see \cite[\S C.6]{BGL}, and also \cite{Sav} for an extension to $\RCD(K,\infty)$-spaces).
Here we give a direct proof by calculations in coordinates.

\begin{proposition}[Improved Bochner inequality]\label{pr:Boc+}
Assume $\Ric_{\infty} \ge K$ for some $K \in \R$.
Then we have, for any $u \in \cC^{\infty}(M)$,
\begin{equation}\label{eq:Boc+}
\Delta\!^{\Grad u} \bigg[ \frac{F^2(\Grad u)}{2} \bigg] -d(\Lap u)(\Grad u) -KF^2(\Grad u)
 \ge d[F(\Grad u)] \big( \nabla^{\Grad u} [F(\Grad u)] \big)
\end{equation}
point-wise on $M_u$.
\end{proposition}

\begin{proof}
By comparing \eqref{eq:Boc} with $N=\infty$ and \eqref{eq:Boc+}, it suffices to show
\begin{equation}\label{eq:Boc+'}
4F^2(\Grad u) \| \Grad^2 u \|_{\HS(\Grad u)}^2 \ge
 d[F^2(\Grad u)] \big( \nabla^{\Grad u} [F^2(\Grad u)] \big).
\end{equation}
Fix $x \in M_u$ and choose local coordinates such that $g_{ij}(\Grad u(x))=\delta_{ij}$.
We first calculate the RHS of \eqref{eq:Boc+'} at $x$ as
\begin{align*}
&d[F^2(\Grad u)] \big( \nabla^{\Grad u} [F^2(\Grad u)] \big)
 = \sum_{i=1}^n \bigg( \frac{\del[F^2(\Grad u)]}{\del x^i} \bigg)^2 \\
&= \sum_{i=1}^n \bigg( \frac{\del}{\del x^i} \bigg[ 
 \sum_{j,k=1}^n g^*_{jk}(du) \frac{\del u}{\del x^j} \frac{\del u}{\del x^k} \bigg] \bigg)^2 \\
&= \sum_{i=1}^n \bigg( 2\sum_{j=1}^n \frac{\del u}{\del x^j} \frac{\del^2 u}{\del x^i \del x^j}
 +\sum_{j,k=1}^n \frac{\del g^*_{jk}}{\del x^i}(du) \frac{\del u}{\del x^j} \frac{\del u}{\del x^k}
 +\sum_{j,k,l=1}^n \frac{\del g_{jk}^*}{\del \alpha_l}(du) \frac{\del^2 u}{\del x^i \del x^l}
 \frac{\del u}{\del x^j} \frac{\del u}{\del x^k} \bigg)^2 \\
&= \sum_{i=1}^n \bigg( 2\sum_{j=1}^n \frac{\del u}{\del x^j} \frac{\del^2 u}{\del x^i \del x^j}
 +\sum_{j,k=1}^n \frac{\del g^*_{jk}}{\del x^i}(du) \frac{\del u}{\del x^j} \frac{\del u}{\del x^k} \bigg)^2,
\end{align*}
where we used Euler's theorem (Theorem~\ref{th:Euler}, similarly to \eqref{eq:Av})
in the last equality.
Next we observe from \eqref{eq:covd} and \eqref{eq:Gamma} that, again at $x$,
\begin{align*}
&\Grad^2 u \bigg( \frac{\del}{\del x^j} \bigg)
 =D^{\Grad u}_{\del/\del x^j}(\Grad u) \\
&= \sum_{i=1}^n \bigg\{ \frac{\del}{\del x^j} \bigg[
 \sum_{k=1}^n g^*_{ik}(du) \frac{\del u}{\del x^k} \bigg]
 +\sum_{k=1}^n \Gamma_{jk}^i(\Grad u) \frac{\del u}{\del x^k} \bigg\} \frac{\del}{\del x^i} \\
&= \sum_{i=1}^n \bigg\{ \frac{\del^2 u}{\del x^j \del x^i}
 +\sum_{k=1}^n \frac{\del g^*_{ik}}{\del x^j}(du) \frac{\del u}{\del x^k}
 +\sum_{k=1}^n \gamma_{jk}^i(\Grad u) \frac{\del u}{\del x^k}
 -\sum_{l=1}^n \frac{A_{ijl}(\Grad u)}{F(\Grad u)} G^l(\Grad u) \bigg\} \frac{\del}{\del x^i} \\
&= \sum_{i=1}^n \bigg\{ \frac{\del^2 u}{\del x^i \del x^j} +\sum_{k=1}^n
 \bigg( \gamma_{jk}^i -\frac{\del g_{ik}}{\del x^j} \bigg)(\Grad u) \frac{\del u}{\del x^k}
 -\sum_{k=1}^n \frac{A_{ijk} G^k}{F}(\Grad u) \bigg\} \frac{\del}{\del x^i}.
\end{align*}
In the last line we used
\[ \frac{\del g^*_{ik}}{\del x^j}\big( du(x) \big) =-\frac{\del g_{ik}}{\del x^j}\big( \Grad u(x) \big). \]
Hence we deduce from the Cauchy--Schwarz inequality, \eqref{eq:Av} and \eqref{eq:gamma} that
\begin{align*}
&F^2(\Grad u) \| \Grad^2 u \|_{\HS(\Grad u)}^2 \\
&= \sum_{j=1}^n \bigg( \frac{\del u}{\del x^j} \bigg)^2 \cdot
 \sum_{i,j=1}^n \bigg( \frac{\del^2 u}{\del x^i \del x^j} +\sum_{k=1}^n
 \bigg( \gamma_{jk}^i -\frac{\del g_{ik}}{\del x^j} \bigg)(\Grad u) \frac{\del u}{\del x^k}
 -\sum_{k=1}^n \frac{A_{ijk} G^k}{F}(\Grad u) \bigg)^2 \\
&\ge \sum_{i=1}^n \bigg( \sum_{j=1}^n \frac{\del u}{\del x^j}
 \bigg\{ \frac{\del^2 u}{\del x^i \del x^j} +\sum_{k=1}^n
 \bigg( \gamma_{jk}^i -\frac{\del g_{ik}}{\del x^j} \bigg)(\Grad u) \frac{\del u}{\del x^k}
 -\sum_{k=1}^n \frac{A_{ijk} G^k}{F}(\Grad u) \bigg\} \bigg)^2 \\
&= \sum_{i=1}^n \bigg( \sum_{j=1}^n \frac{\del u}{\del x^j} \frac{\del^2 u}{\del x^i \del x^j}
 +\sum_{j,k=1}^n \bigg( \gamma_{jk}^i -\frac{\del g_{ik}}{\del x^j} \bigg)(\Grad u)
 \frac{\del u}{\del x^j} \frac{\del u}{\del x^k} \bigg)^2 \\
&= \sum_{i=1}^n \bigg( \sum_{j=1}^n \frac{\del u}{\del x^j} \frac{\del u^2}{\del x^i \del x^j}
 -\frac{1}{2} \sum_{j,k=1}^n \frac{\del g_{jk}}{\del x^i}(\Grad u)
 \frac{\del u}{\del x^j} \frac{\del u}{\del x^k} \bigg)^2.
\end{align*}
This completes the proof of \eqref{eq:Boc+'} as well as \eqref{eq:Boc+}.
$\qedd$
\end{proof}

The following integrated form can be shown in the same way as Theorem~\ref{th:BWint},
we refer to \cite[Theorem~3.6]{OSbw} for details.

\begin{corollary}[Integrated form]\label{cr:Boc+}
Assume $\Ric_{\infty} \ge K$ for some $K \in \R$.
Given $u \in H^2_{\loc}(M) \cap \cC^1(M)$ such that $\Lap u \in H_{\loc}^1(M)$,
we have
\begin{align*}
&-\int_M d\phi \bigg( \nabla^{\Grad u} \bigg[ \frac{F^2(\Grad u)}{2} \bigg] \bigg) \,d\fm \\
&\ge \int_M \phi \Big\{ d(\Lap u)(\Grad u) +K F^2(\Grad u)
 +d[F(\Grad u)]\big( \nabla^{\Grad u}[F(\Grad u)] \big) \Big\} \,d\fm
\end{align*}
for all bounded nonnegative functions $\phi \in H_c^1(M) \cap L^{\infty}(M)$.
\end{corollary}

\subsection{$L^1$-gradient estimate}\label{ssc:L1}%%%%%%%%%%
%%%%%%%%%%%%%%%%%%%%%%

The improved Bochner inequality \eqref{eq:Boc+} yields the following
\emph{$L^1$-gradient estimate}, under a technical (likely redundant) assumption that
$d[F(\Grad u_t)](\nabla^{\Grad u_t}[F(\Grad u_t)]) \in L^1(M)$ for all $t>0$,
which holds in the compact case thanks to the $H^2_{\loc}$-regularity
(recall Theorem~\ref{th:hf}).

%%%%%%%%%%%%%%%
%%%%%%%%%%%%%%%

\begin{theorem}[$L^1$-gradient estimate]\label{th:L1}
Let $(M,F,\fm)$ be complete and satisfy $\Ric_{\infty} \ge K$, $\sC_F<\infty$ and $\sS_F<\infty$,
and $(u_t)_{t \ge 0}$ be a global solution to the heat equation with $u_0 \in \cC^{\infty}_c(M)$.
We further assume that
\begin{equation}\label{eq:hypo}
d[F(\Grad u_t)] \big( \nabla^{\Grad u_t} [F(\Grad u_t)] \big) \in L^1(M)
\end{equation}
for all $t>0$.
Then we have
\[ F\big( \Grad u_t (x) \big)
 \le \e^{-K(t-s)} P_{s,t}^{\Grad u} \big( F(\Grad u_s) \big)(x) \]
for all $0 \le s<t<\infty$ and $x \in M$.
\end{theorem}

\begin{proof}
Notice first that $F(\Grad u_0) \in H^1_0(M) \cap L^{\infty}(M)$ since $u_0 \in \cC^{\infty}_c(M)$.
Fix arbitrary $\ve>0$ and let us consider the function
\[ \xi_{\sigma}:=\sqrt{\e^{-2K\sigma} F^2(\Grad u_{t-\sigma}) +\ve},
 \qquad 0 \le \sigma \le t-s. \]
Note from the proof of \cite[Theorem~4.1]{OSbw} that
\begin{equation}\label{eq:delF}
\frac{\del}{\del \sigma} \bigg[ \frac{F^2(\Grad u_{t-\sigma})}{2} \bigg]
 = -\frac{\del}{\del t} \bigg[ \frac{F^2(\Grad u_{t-\sigma})}{2} \bigg]
 =-d(\Lap u_{t-\sigma})(\Grad u_{t-\sigma}).
\end{equation}
Hence we have, on one hand,
\[ \del_{\sigma} \xi_{\sigma} =-\frac{\e^{-2K\sigma}}{\xi_{\sigma}}
 \big\{ KF^2(\Grad u_{t-\sigma}) +d(\Lap u_{t-\sigma})(\Grad u_{t-\sigma}) \big\}. \]
On the other hand, for any nonnegative function $\phi \in \cC^{\infty}_c(M)$, we observe
\begin{align*}
&\int_M d\phi(\nabla^{\Grad u_{t-\sigma}} \xi_{\sigma}) \,d\fm
 =\int_M \frac{\e^{-2K\sigma}}{\xi_{\sigma}}
 d\phi\bigg( \nabla^{\Grad u_{t-\sigma}} \bigg[ \frac{F^2(\Grad u_{t-\sigma})}{2} \bigg] \bigg) \,d\fm \\
&=\int_M \bigg( d\bigg( \phi \frac{\e^{-2K\sigma}}{\xi_{\sigma}} \bigg)
 +\phi \frac{\e^{-2K\sigma}}{\xi_{\sigma}^2} d\xi_{\sigma} \bigg)
 \bigg( \nabla^{\Grad u_{t-\sigma}} \bigg[ \frac{F^2(\Grad u_{t-\sigma})}{2} \bigg] \bigg) \,d\fm \\
&= \int_M d\bigg( \phi \frac{\e^{-2K\sigma}}{\xi_{\sigma}} \bigg)
 \bigg( \nabla^{\Grad u_{t-\sigma}} \bigg[ \frac{F^2(\Grad u_{t-\sigma})}{2} \bigg] \bigg) \,d\fm \\
&\quad +\int_M \phi \frac{\e^{-4K\sigma}}{\xi_{\sigma}^3}
 d\bigg[ \frac{F^2(\Grad u_{t-\sigma})}{2} \bigg]
 \bigg( \nabla^{\Grad u_{t-\sigma}} \bigg[ \frac{F^2(\Grad u_{t-\sigma})}{2} \bigg] \bigg) \,d\fm \\
&\le \int_M d\bigg( \phi \frac{\e^{-2K\sigma}}{\xi_{\sigma}} \bigg)
 \bigg( \nabla^{\Grad u_{t-\sigma}} \bigg[ \frac{F^2(\Grad u_{t-\sigma})}{2} \bigg] \bigg) \,d\fm \\
&\quad +\int_M \phi \frac{\e^{-2K\sigma}}{\xi_{\sigma}}
 d[F(\Grad u_{t-\sigma})]\big( \nabla^{\Grad u_{t-\sigma}}[F(\Grad u_{t-\sigma})] \big) \,d\fm,
\end{align*}
where we used $F^2(\Grad u_{t-\sigma}) \le \e^{2K\sigma} \xi_{\sigma}^2$ in the last inequality.
Therefore the improved Bochner inequality (Corollary~\ref{cr:Boc+}) shows that
\begin{equation}\label{eq:3.2.7}
\Delta\!^{\Grad u_{t-\sigma}} \xi_{\sigma} +\del_{\sigma} \xi_{\sigma} \ge 0
\end{equation}
in the weak sense.
Notice that the test function $\phi$ can be in fact taken from
$H^1_0(M) \cap L^{\infty}(M)$ thanks to the hypothesis \eqref{eq:hypo}
and $\xi_{\sigma} \ge \sqrt{\ve}$.

For a nonnegative function $\varphi \in \cC^{\infty}_c(M)$ and $\sigma \in (0,t-s)$, set
\[ \Phi(\sigma):=\int_M \varphi \cdot P^{\Grad u}_{t-\sigma,t}(\xi_{\sigma}) \,d\fm
 =\int_M \widehat{P}^{\Grad u}_{t-\sigma,t}(\varphi) \cdot \xi_{\sigma} \,d\fm. \]
We deduce from \eqref{eq:Phat} and \eqref{eq:f1f2} that
\begin{align*}
\Phi'(\sigma) &=
 \int_M \widehat{P}^{\Grad u}_{t-\sigma,t}(\varphi) \cdot \del_{\sigma} \xi_{\sigma} \,d\fm
 -\int_M d\xi_\sigma
 \big( \nabla^{\Grad u_{t-\sigma}} \big[ \widehat{P}^{\Grad u}_{t-\sigma,t}(\varphi) \big] \big)
 \,d\fm \\
&= \int_M \widehat{P}^{\Grad u}_{t-\sigma,t}(\varphi) \cdot \del_{\sigma} \xi_{\sigma} \,d\fm
 -\int_M d[\widehat{P}^{\Grad u}_{t-\sigma,t}(\varphi)]
 (\nabla^{\Grad u_{t-\sigma}} \xi_{\sigma}) \,d\fm.
% \label{eq:Phi'}
\end{align*}
Therefore we can apply \eqref{eq:3.2.7} with the test function
$\widehat{P}^{\Grad u}_{t-\sigma,t}(\varphi)$ (thanks to Proposition~\ref{pr:lin})
to obtain $\Phi'(\sigma) \ge 0$.
This implies
\[ \int_M \varphi \cdot \xi_0 \,d\fm \le \int_M \varphi \cdot P^{\Grad u}_{s,t}(\xi_{t-s}) \,d\fm. \]
By the arbitrariness of $\varphi$ and $\ve$, we have
\[ F(\Grad u_t) \le \e^{-K(t-s)} P^{\Grad u}_{s,t}\big( F(\Grad u_s) \big) \]
almost everywhere.
Since both sides are H\"older continuous (Proposition~\ref{pr:lin}(ii)),
this completes the proof.
%%%%%%%%
%%%%%%%
$\qedd$
\end{proof}

It is a standard fact that the $L^1$-gradient estimate implies the $L^2$-bound.

\begin{corollary}[$L^2$-gradient estimate, noncompact case]\label{cr:L2}
Let $(M,F,\fm)$ be complete and satisfy $\Ric_{\infty} \ge K$, $\sC_F<\infty$ and $\sS_F<\infty$,
and $(u_t)_{t \ge 0}$ be a global solution to the heat equation with $u_0 \in \cC^{\infty}_c(M)$
and satisfying \eqref{eq:hypo} for all $t>0$.
Then we have
\[ F^2\big( \Grad u_t (x) \big)
 \le \e^{-2K(t-s)} P_{s,t}^{\Grad u} \big( F^2(\Grad u_s) \big)(x) \]
for all $0 \le s<t<\infty$ and $x \in M$.
\end{corollary}

\begin{proof}
This is a consequence of a kind of Jensen's inequality:
%\begin{equation}%\label{eq:Jensen}
\[ P^{\Grad u}_{s,t}(f)^2 \le P^{\Grad u}_{s,t}(f^2) \]
%\end{equation}
for $f \in L^2(M) \cap L^{\infty}(M)$.
For $\psi \in \cC^{\infty}_c(M)$ with $0 \le \psi \le 1$ and $r \in \R$, we have
\begin{align*}
0 &\le P^{\Grad u}_{s,t}\big( (rf+\psi)^2 \big)
 =r^2 P^{\Grad u}_{s,t}(f^2) +2r P^{\Grad u}_{s,t}(f\psi) +P^{\Grad u}_{s,t}(\psi^2) \\
&\le r^2 P^{\Grad u}_{s,t}(f^2) +2r P^{\Grad u}_{s,t}(f\psi) +1.
\end{align*}
Letting $f\psi \to f$ in $L^2(M)$, we find
$r^2 P^{\Grad u}_{s,t}(f^2) +2r P^{\Grad u}_{s,t}(f) +1 \ge 0$ for all $r \in \R$.
Hence $P^{\Grad u}_{s,t}(f)^2 -P^{\Grad u}_{s,t}(f^2) \le 0$
as desired.
$\qedd$
\end{proof}

\subsection{On the hypothesis \eqref{eq:hypo}}\label{ssc:hypo}%%%%%%%%%%%%
%%%%%%%%%%%%%%%%

The hypothesis \eqref{eq:hypo} seems redundant and indeed unnecessary
for weighted Riemannian manifolds and $\RCD$-spaces.
Especially, when $K>0$, the Gaussian decay of the measure (\cite[Theorem~4.26]{StI})
could imply \eqref{eq:hypo}.
Let us give some more comments on \eqref{eq:hypo}.

\subsubsection{Weighted Riemannian case}%%%%%%%%%%%%

We essentially followed the proof of \cite[Theorem~3.2.4]{BGL} in Theorem~\ref{th:L1}.
Then we have
\[ d\xi_{\sigma} (\nabla^{\Grad u_{t-\sigma}} \xi_{\sigma})
 \le \e^{-2K\sigma}
 d[F(\Grad u_{t-\sigma})]\big( \nabla^{\Grad u_{t-\sigma}}[F(\Grad u_{t-\sigma})] \big), \]
and the improved Bochner inequality (Proposition~\ref{pr:Boc+}) implies
\[ \int_M d[F(\Grad u)] \big( \nabla^{\Grad u} [F(\Grad u)] \big) \,d\fm
 \le \|\Lap u\|_{L^2}^2 +|K| \cdot \|u\|_{L^2} \|\Lap u\|_{L^2} \]
for $u \in \cC^{\infty}_c(M)$.
Now in \cite{BGL}, for a linear operator $\mathrm{L}$,
we make use of the density of $\mathcal{A}_0=\cC^{\infty}_c(M)$
in the domain $\mathcal{D}(\mathrm{L})$ with respect to the norm
\[ \|f\|_{\mathcal{D}(\mathrm{L})} :=\big( \|f\|_{L^2}^2 +\|\mathrm{L}f\|_{L^2}^2 \big)^{1/2} \]
to extend the above estimate to $\mathcal{D}(\mathrm{L})$.
This density is a consequence of the \emph{hypo-ellipticity} (see \cite[Proposition~3.2.1]{BGL}),
which is defined by the property that any solution to $\mathrm{L}^*f=\lambda f$ is smooth
(see also \cite[Definition~3.3.8]{BGL}, typically $\mathcal{A}=\cC^{\infty}(M)$).
This is not the case for operators with nonsmooth coefficients,
thereby it is unclear if we can apply this method in the Finsler case
(to the linearized Laplacian $\Delta\!^{\Grad u}$).

\if0%%%%
Precisely, $\mathrm{L}^*$ is the adjoint of $\mathrm{L}$ defined by
\[ \int_M \phi \cdot \mathrm{L}^*f \,d\fm =\int_M f \cdot \mathrm{L} \phi \,d\fm \]
for all $\phi \in \cC^{\infty}_c(M)$, where $f,\mathrm{L}^*f \in L^2(M)$,
and $\mathcal{D}(\mathrm{L}^*) \supset \mathcal{D}(\mathrm{L})$.
Then the above density of $\cC^{\infty}(M)$ in $\mathcal{D}(\mathrm{L})$
($\cC^{\infty}_c(M)$ is a \emph{core} of $\mathcal{D}(\mathrm{L})$)
is equivalent to the \emph{essential self-adjointness}:
$\mathcal{D}(\mathrm{L}^*)=\mathcal{D}(\mathrm{L})$ (\cite[\S 3.1.8]{BGL}).

In order to follow the proof of \cite[Proposition~3.2.1]{BGL},
it is in fact enough to show that: any solution to $\mathrm{L}^*f=f$ is zero
for $\mathrm{L}=\Delta\!^{\Grad u}$.
Namely, $f \in L^2(M)$ satisfying
\[ \int_M f \cdot \Delta\!^{\Grad u} \phi \,d\fm =\int_M f \phi \,d\fm \]
for all $\phi \in \cC^{\infty}_c(M)$ is zero.
Note that $\Delta\!^{\Grad u}$ is a nonpositive operator and hence
the consequence is clear if $f \in \mathcal{D}(\mathrm{L})$,
but a priori $f \in \mathcal{D}(\mathrm{L}^*)$.
\fi%%%%

\subsubsection{$\RCD$-case}%%%%%%%%%%%%

In $\RCD(K,\infty)$-spaces, we obtain the \emph{Wasserstein contraction estimate} of heat flow
by the convexity of the relative entropy,
and then the gradient estimates follow by the duality argument.
Moreover, we can obtain the Bochner inequality by differentiating the gradient estimate
(see \cite{AGSrcd,AGSboc,GKO,Sav} for details).

This method could avoid the use of the functional analytic argument
involving $\mathcal{A}_0$ and $\mathcal{A}$,
and what is important and interesting here is that
the Bochner inequality derived from the gradient estimate is of the form:
\[ \int \Delta \phi \cdot \frac{|\nabla u|^2}{2} \,d\fm
 \ge \int \phi \big\{ d(\Delta u)(\nabla u) +K|\nabla u|^2 \big\} \,d\fm \]
for $u \in \mathcal{D}(\Delta)$ with $\Delta u \in H^1$
and $\phi \in \mathcal{D}(\Delta) \cap L^{\infty}$ with $\Delta \phi \in L^{\infty}$.
In the LHS, what we have directly from the point-wise Bochner inequality is
\[  \int \phi \cdot \Delta \bigg[ \frac{|\nabla u|^2}{2} \bigg] \,d\fm, \]
and modifying this into the above LHS requires an approximation of $\phi$
by functions $\phi_k$ in $\cC^{\infty}_c$ such that $\Delta \phi_k \to \Delta \phi$,
namely the density of $\cC^{\infty}_c$ in the $\mathcal{D}(\Delta)$-norm
as in the approach of \cite{BGL}.

In the Finsler case, we know that the Wasserstein contraction fails
(see Remark~\ref{rm:cont} below).
Nonetheless, if one can show the Bochner inequality in the above form
as well as $F(\Grad u) \in L^{\infty}(M)$, then it follows from the argument along \cite[Lemma~3.2]{Sav}
that $d[F(\Grad u)](\nabla^{\Grad u} F(\Grad u)) \in L^1(M)$
and we obtain the gradient estimates.

%%%%%%
%%%%%%

\subsection{Characterizations of lower Ricci curvature bounds}\label{ssc:char}%%%%%%%%%%%%
%%%%%%%%%%%%%%%%

We close the section with several characterizations of the lower Ricci curvature bound
$\Ric_{\infty} \ge K$.

\begin{theorem}[Characterizations of $\Ric_{\infty} \ge K$]\label{th:char}
Let $(M,F,\fm)$ be complete and satisfy $\sC_F<\infty$ and $\sS_F<\infty$.
We assume that \eqref{eq:hypo} holds for all solutions $(u_t)_{t \ge 0}$
to the heat equation with $u_0 \in \cC^{\infty}_c(M)$.
Then, for each $K \in \R$, the following are equivalent$:$
\begin{enumerate}[{\rm (I)}]
\item $\Ric_{\infty} \ge K$.

\item The Bochner inequality
\[ \Delta\!^{\Grad u} \bigg[ \frac{F^2(\Grad u)}{2} \bigg] -d(\Lap u)(\Grad u)
 \ge KF^2(\Grad u) \]
holds on $M_u$ for all $u \in \cC^{\infty}(M)$.

\item The improved Bochner inequality
\[ \Delta\!^{\Grad u} \bigg[ \frac{F^2(\Grad u)}{2} \bigg] -d(\Lap u)(\Grad u) -KF^2(\Grad u)
 \ge d[F(\Grad u)] \big( \nabla^{\Grad u} [F(\Grad u)] \big) \]
holds on $M_u$ for all $u \in \cC^{\infty}(M)$.

\item The $L^2$-gradient estimate
\[ F^2(\Grad u_t)
 \le \e^{-2K(t-s)} P_{s,t}^{\Grad u} \big( F^2(\Grad u_s) \big), \qquad 0 \le s<t<\infty, \]
holds for all global solutions $(u_t)_{t \ge 0}$ to the heat equation
with $u_0 \in \cC^{\infty}_c(M)$.

\item The $L^1$-gradient estimate
\[ F(\Grad u_t)
 \le \e^{-K(t-s)} P_{s,t}^{\Grad u} \big( F(\Grad u_s) \big), \qquad 0 \le s<t<\infty, \]
holds for all global solutions $(u_t)_{t \ge 0}$ to the heat equation
with $u_0 \in \cC^{\infty}_c(M)$.
\end{enumerate}
\end{theorem}

\begin{proof}
We have shown (I) $\Rightarrow$ (III) in Proposition~\ref{pr:Boc+},
(III) $\Rightarrow$ (V) in Theorem~\ref{th:L1}, and (V) $\Rightarrow$ (IV) in Corollary~\ref{cr:L2}.
One can deduce (IV) $\Rightarrow$ (II) from the proof of \cite[Theorem~4.1]{OSbw}
or by differentiating $F^2(\Grad u_t) \le \e^{-2Kt} P_{0,t}^{\Grad u}(F^2(\Grad u_0))$ at $t=0$
(recall \eqref{eq:delF}, see also \cite{GKO}).
Let us finally prove (II) $\Rightarrow$ (I).
Given $v_0 \in T_{x_0}M \setminus 0$,
fix local coordinates $(x^i)_{i=1}^n$ around $x_0$ with $g_{ij}(v_0)=\delta_{ij}$
and $x^i(x_0)=0$ for all $i$.
Consider the function
\[ u:=\sum_{i=1}^n v_0^i x^i
 +\frac{1}{2}\sum_{i,j,k=1}^n \Gamma_{ij}^k (v_0) v_0^k x^i x^j \]
on a neighborhood of $x_0$, and observe that
$\Grad u(x_0)=v_0$ as well as $(\Grad^2 u)|_{T_{x_0}M}=0$
(see \cite[Lemma~2.3]{OSbw} for the precise expression in coordinates of $\Grad^2 u$).
Then the Bochner--Weitzenb\"ock formula \eqref{eq:BW} and (II) imply
\[ \Ric_{\infty}(v_0)
 =\Delta\!^{\Grad u} \bigg[ \frac{F^2(\Grad u)}{2} \bigg](x_0) -d(\Lap u)(\Grad u)(x_0)
 \ge KF^2(v_0). \]
This completes the proof.
$\qedd$
\end{proof}

\begin{remark}[The lack of contraction]\label{rm:cont}
In the Riemannian context, lower Ricci curvature bounds are also equivalent to
contraction estimates of heat flow with respect to the Wasserstein distance
(we refer to \cite{vRS} for the Riemannian case,
and \cite{EKS} for the case of $\RCD$-spaces).
More generally, for linear semigroups, gradient estimates are directly equivalent to
the corresponding contraction properties (see \cite{Ku}).
In our Finsler setting, however, the lack of the \emph{commutativity} (see \cite{OP})
prevents such a contraction estimate, at least in the same form (see \cite{OSnc} for details).
\end{remark}

\begin{remark}[Similarities to (super) Ricci flow theory]\label{rm:sRF}
The methods in this section have connections with the Ricci flow theory.
Ricci flow provides time-dependent Riemannian metrics obeying a kind of heat equation
on the space of Riemannian metrics,
while we considered the time-dependent (singular) Riemannian structures $g_{\Grad u}$
for $u$ solving the heat equation.
More precisely, what corresponds to our lower Ricci curvature bound is
\emph{super Ricci flow} (super-solutions to the Ricci flow equation).
We refer to \cite{MT} for an inspiring work
on a characterization of super Ricci flow in terms of the contraction of heat flow,
and to \cite{Stsrf} for a recent investigation of super Ricci flow on time-dependent
metric measure spaces including various characterizations related to Theorem~\ref{th:char}.
Then, again, what is missing in our Finsler setting is the contraction property,
for which the Riemannian nature of the space is necessary.
\end{remark}

\section{Bakry--Ledoux's isoperimetric inequality}\label{sc:BL}%%%%%%%%%%%%%%%%%
%%%%%%%%%%%%%%%%

This section is devoted to the isoperimetric inequality,
as a geometric application of the improved Bochner inequality (Proposition~\ref{pr:Boc+}).
We will assume $\Ric_{\infty} \ge K>0$, then $\fm(M)<\infty$ holds
(see \cite[Theorem~4.26]{StI})
and hence we can normalize $\fm$ as $\fm(M)=1$ without changing $\Ric_{\infty}$
($c\fm$ with $c>0$ gives the same weighted Ricci curvature as $\fm$).

For a Borel set $A \subset M$,
define the \emph{Minkowski exterior boundary measure} as
\[ \fm^+(A):=\liminf_{\ve \downarrow 0} \frac{\fm(B^+(A,\ve))-\fm(A)}{\ve}, \]
where $B^+(A,\ve):=\{y \in M \,|\, \inf_{x \in A}d(x,y)<\ve \}$
is the forward $\ve$-neighborhood of $A$.
Then the (forward) \emph{isoperimetric profile} $\cI_{(M,F,\fm)}:[0,1] \lra [0,\infty)$
of $(M,F,\fm)$ is defined by
\[ \cI_{(M,F,\fm)}(\theta):=\inf\{ \fm^+(A) \,|\,
 A \subset M: \text{Borel set with}\ \fm(A)=\theta \}. \]
Clearly $\cI_{(M,F,\fm)}(0)=\cI_{(M,F,\fm)}(1)=0$.
The following is our main result (stated as Theorem in the introduction).

\begin{theorem}[Bakry--Ledoux's isoperimetric inequality]\label{th:BL}
Let $(M,F)$ be complete and satisfy $\Ric_{\infty} \ge K>0$,
$\fm(M)=1$, $\sC_F<\infty$ and $\sS_F<\infty$.
We assume that \eqref{eq:hypo} holds for all solutions $(u_t)_{t \ge 0}$
to the heat equation with $u_0 \in \cC^{\infty}_c(M)$.
Then we have
\begin{equation}\label{eq:BLisop}
\cI_{(M,F,\fm)}(\theta) \ge \cI_K(\theta)
\end{equation}
for all $\theta \in [0,1]$, where
\[ \cI_K(\theta):=\sqrt{\frac{K}{2\pi}} \e^{-Kc^2(\theta)/2} \qquad
 \text{with}\ \ \theta=\int_{-\infty}^{c(\theta)} \sqrt{\frac{K}{2\pi}} \e^{-Ka^2/2} \,da. \]
\end{theorem}

Recall that, under $\sC_F<\infty$ or $\sS_F<\infty$, the forward completeness
is equivalent to the backward completeness by Lemma~\ref{lm:rev}.
In the Riemannian case, the inequality \eqref{eq:BLisop} is due to Bakry and Ledoux \cite{BL}
(see also \cite[\S 8.5.2]{BGL})
and can be regarded as the dimension-free version of \emph{L\'evy--Gromov's isoperimetric inequality}
(see \cite{Lev1,Lev2,Gr}).
L\'evy--Gromov's classical isoperimetric inequality asserts that
the isoperimetric profile of an $n$-dimensional Riemannian manifold $(M,g)$ with $\Ric \ge n-1$
is bounded below by the profile of the unit sphere $\Sph^n$
(both spaces are equipped with the normalized volume measures).
In \eqref{eq:BLisop}, the role of the unit sphere is played by the real line $\R$
equipped with the Gaussian measure $\sqrt{K/2\pi} \,\e^{-Kx^2/2} \,dx$,
thereby \eqref{eq:BLisop} is also called the \emph{Gaussian isoperimetric inequality}.

In \cite{Oneedle}, generalizing Cavalletti and Mondino's localization technique in \cite{CM},
we showed the slightly weaker inequality (recall the introduction)
\[ \cI_{(M,F,\fm)}(\theta) \ge \Lambda_F^{-1} \cdot \cI_K(\theta) \]
under the finite reversibility $\Lambda_F<\infty$ (but without $\sC_F<\infty$ nor $\sS_F<\infty$).
In fact we have treated in \cite{Oneedle} the general curvature-dimension-diameter bound
$\Ric_N \ge K$ and $\diam M \le D$ (in accordance with \cite{Misharp}).
Theorem~\ref{th:BL} sharpens the estimate in \cite{Oneedle}
in the special case of $N=D=\infty$ and $K>0$.

\subsection{Ergodicity}\label{ssc:ergo}%%%%%%%%%%%%%%%%%
%%%%%%%%%%%%%%%%%%%%%%

We begin with some properties induced from our hypothesis $\Ric_{\infty} \ge K>0$.

\begin{lemma}[Global Poincar\'e inequality]\label{lm:Poin}
Suppose that $(M,F,\fm)$ is forward or backward complete,
$\Ric_{\infty} \ge K>0$ and $\fm(M)=1$.
Then we have, for any locally Lipschitz function $f \in H_0^1(M)$,
\begin{equation}\label{eq:Poin}
\int_M f^2 \,d\fm -\bigg( \int_M f \,d\fm \bigg)^2
 \le \frac{1}{K} \int_M F^*(df)^2 \,d\fm.
\end{equation}
\end{lemma}

\begin{proof}
It is well known that the curvature bound $\Ric_{\infty} \ge K$ (or $\CD(K,\infty)$) implies
the \emph{log-Sobolev inequality},
\begin{equation}\label{eq:logSob}
\int_M \rho \log \rho \,d\fm \le \frac{1}{2K} \int_M \frac{F^*(d\rho)^2}{\rho} \,d\fm
\end{equation}
for nonnegative locally Lipschitz functions $\rho$ with $\int_M \rho \,d\fm=1$,
and that \eqref{eq:Poin} follows from \eqref{eq:logSob}
(see \cite{OV,LV,Vi,Oint}).
Here we explain the latter step for thoroughness.

By truncation, let us assume that $f$ is bounded.
Since
\[ \int_M f^2 \,d\fm -\bigg( \int_M f \,d\fm \bigg)^2
 =\int_M \bigg( f -\int_M f \,d\fm \bigg)^2 \,d\fm, \]
we can further assume that $\int_M f \,d\fm=0$.
There is nothing to prove if $f \equiv 0$, thereby assume $\|f\|_{L^{\infty}}>0$.
For $\ve \in \R $ with $|\ve|<\|f\|_{L^{\infty}}^{-1}$,
we consider the probability measure $(1+\ve f) \fm$.
Then the log-Sobolev inequality for $\rho_{\ve}:=1+\ve f$ under $\Ric_{\infty} \ge K$ implies
\[ \int_M (1+\ve f) \log(1+\ve f) \,d\fm
 \le \frac{1}{2K} \int_M \frac{\ve^2 F^*(df)^2}{1+\ve f} \,d\fm. \]
Expanding the LHS at $\ve=0$ yields
\[ \int_M \bigg\{ \ve f +\frac{1}{2}(\ve f)^2 +O(\ve^3) \bigg\} \,d\fm
 =\frac{\ve^2}{2} \int_M f^2 \,d\fm +O(\ve^3), \]
where $O(\ve^3)$ in the LHS is uniform in $M$
thanks to the boundedness of $f$.
Hence we have
\[ \frac{\ve^2}{2} \int_M f^2 \,d\fm
 \le \frac{1}{1-\ve \|f\|_{L^{\infty}}} \frac{\ve^2}{2K} \int_M F^*(df)^2 \,d\fm +O(\ve^3). \]
Dividing both sides by $\ve^2$ and letting $\ve \to 0$ implies \eqref{eq:Poin}.
$\qedd$
\end{proof}

The LHS of \eqref{eq:Poin} is the \emph{variance} of $f$:
\[ \Var_{\fm}(f) :=\int_M f^2 \,d\fm -\bigg( \int_M f \,d\fm \bigg)^2. \]
We next show that the Poincar\'e inequality \eqref{eq:Poin} yields
the exponential decay of the variance and a kind of \emph{ergodicity} along heat flow
(similarly to \cite[\S 4.2]{BGL}),
which is one of the key ingredients in the proof of Theorem~\ref{th:BL}
(see the proof of Corollary~\ref{cr:key}).
Given a global solution $(u_t)_{t \ge 0}$ to the heat equation,
since the finiteness of the total mass together with
$\Lambda_F<\infty$ and the completeness implies $1 \in H_0^1(M)$,
we observe the mass conservation:
\begin{equation}\label{eq:consv}
\int_M P^{\Grad u}_{s,t}(f) \,d\fm =\int_M f \,d\fm
\end{equation}
for any $f \in H_0^1(M)$ and $0 \le s<t<\infty$.

\begin{proposition}[Variance decay and ergodicity]\label{pr:ergo}
Assume that $(M,F,\fm)$ is complete and satisfies $\sC_F<\infty$, $\sS_F<\infty$,
$\Ric_{\infty} \ge K>0$ and $\fm(M)=1$.
Then we have, given any global solution $(u_t)_{t \ge 0}$ to the heat equation
and $f \in H_0^1(M)$,
\[ \Var_{\fm}\! \big( P^{\Grad u}_{s,t}(f) \big) \le \e^{-2K(t-s)/\sS_F} \Var_{\fm}(f) \]
for all $0 \le s<t<\infty$.
In particular, $P^{\Grad u}_{s,t}(f)$ converges to
the constant function $\int_M f \,d\fm$ in $L^2(M)$ as $t \to \infty$.
\end{proposition}

\begin{proof}
Put $f_t:=P^{\Grad u}_{s,t}(f)$, then $\int_M f_t \,d\fm=\int_M f \,d\fm$
holds by \eqref{eq:consv}.
It follows from Lemmas~\ref{lm:us}, \ref{lm:Poin} that
\begin{align*}
\frac{d}{dt}\big[\! \Var_{\fm}(f_t) \big]
&=-2 \int_M df_t(\nabla^{V_t} f_t) \,d\fm
 =-2 \int_M g^*_{\cL(V_t)}(df_t,df_t) \,d\fm \\
&\le -\frac{2}{\sS_F} \int_M F^*(df_t)^2 \,d\fm
 \le -\frac{2K}{\sS_F} \Var_{\fm}(f_t).
\end{align*}
Hence $\e^{2Kt/\sS_F} \Var_{\fm}(f_t)$ is non-increasing in $t$,
this completes the proof of the first assertion.
Then the second assertion is straightforward since
\[ \Var_{\fm}(f_t)=\int_M \bigg( f_t -\int_M f \,d\fm \bigg)^2 d\fm
 \to 0 \quad (t \to \infty). \]
$\qedd$
\end{proof}

\subsection{Key estimate}\label{ssc:key}%%%%%%%%%%%%%%%%%%
%%%%%%%%%%%%%%%%%%%

We next prove a key estimate which would have further applications (see \cite{BL}).
Define
\begin{align*}
\varphi(c) &:= \frac{1}{\sqrt{2\pi}} \int_{-\infty}^c \e^{-b^2/2} \,db,
 \quad c \in \R, \\
\scN(\theta) &:= \varphi' \big( \varphi^{-1}(\theta) \big)
 =\frac{\e^{-\varphi^{-1}(\theta)^2/2}}{\sqrt{2\pi}},
 \quad \theta \in (0,1).
\end{align*}
We set also $\scN(0)=\scN(1):=0$.
Observe that $\scN'=-\varphi^{-1}$ and $\scN''=-1/\scN$ on $(0,1)$.

\begin{theorem}\label{th:key}
Assume that $(M,F,\fm)$ is complete and satisfies
$\Ric_{\infty} \ge K$ for some $K \in \R$, $\sC_F<\infty$, $\sS_F<\infty$ and $\fm(M)<\infty$.
Then, given a global solution $(u_t)_{t \ge 0}$ to the heat equation
with $u_0 \in \cC^{\infty}_c(M)$, $0 \le u_0 \le 1$ and satisfying \eqref{eq:hypo},
we have
\begin{equation}\label{eq:key}
\sqrt{\scN^2(u_t) +\alpha F^2(\Grad u_t)}
 \le P^{\Grad u}_{0,t} \Big( \sqrt{\scN^2(u_0) +c_{\alpha}(t) F^2(\Grad u_0)} \Big)
\end{equation}
on $M$ for all $\alpha \ge 0$ and $t>0$, where
\[ c_{\alpha}(t):=\frac{1-\e^{-2Kt}}{K} +\alpha \e^{-2Kt} >0 \]
and $c_{\alpha}(t):=2t+\alpha$ when $K=0$.
\end{theorem}

For simplicity, we suppressed the dependence of $c_{\alpha}$ on $K$.

\begin{proof}
By replacing $u_0$ with $(1-2\ve)u_0 +\ve$,
we can assume $\ve \le u_0 \le 1-\ve$ for some $\ve>0$,
and then we have $\ve \le u_t \le 1-\ve$ for all $t>0$ (recall \eqref{eq:u>=0}).
Fix $t>0$ and put
\[ \zeta_s:=\sqrt{\scN^2(u_s) +c_{\alpha}(t-s) F^2(\Grad u_s)},
 \qquad 0 \le s \le t \]
(compare this function with $\xi_{\sigma}$ in the proof of Theorem~\ref{th:L1}).
Then \eqref{eq:key} is written as $\zeta_t \le P^{\Grad u}_{0,t}(\zeta_0)$
and it suffices to show $\del_s[P^{\Grad u}_{s,t}(\zeta_s)] \le 0$ in the weak sense.
Observe from \eqref{eq:adj} and \eqref{eq:Phat} that,
for any nonnegative $\phi \in \cC_c^{\infty}((0,t) \times M)$,
\begin{align}
\int_0^t \int_M \del_s \phi_s \cdot P^{\Grad u}_{s,t}(\zeta_s) \,d\fm \,ds
&= \int_0^t \int_M \widehat{P}^{\Grad u}_{s,t}(\del_s \phi_s) \cdot \zeta_s \,d\fm \,ds \nonumber\\
&= \int_0^t \int_M \big\{ \del_s[\widehat{P}^{\Grad u}_{s,t}(\phi_s)]
 +\Delta\!^{V_s}[\widehat{P}^{\Grad u}_{s,t}(\phi_s)] \big\} \cdot \zeta_s \,d\fm \,ds \nonumber\\
&= \int_0^t \int_M \widehat{P}^{\Grad u}_{s,t}(\phi_s) \cdot
 (\Delta\!^{\Grad u_s} \zeta_s -\del_s \zeta_s) \,d\fm \,ds, \label{eq:key-}
\end{align}
where in the second equality we deduce from the linearity of $ \widehat{P}^{\Grad u}_{s,t}$ that
\begin{align*}
&\int_M \del_s[\widehat{P}^{\Grad u}_{s,t}(\phi_s)] \cdot \zeta_s \,d\fm \\
&= \lim_{\ve \to 0} \frac{1}{\ve} \int_M
 \big\{ \widehat{P}^{\Grad u}_{s+\ve,t}(\phi_{s+\ve}) -\widehat{P}^{\Grad u}_{s,t}(\phi_{s+\ve}) \big\}
 \cdot \zeta_s \,d\fm
 +\int_M \widehat{P}^{\Grad u}_{s,t}(\del_s \phi_s) \cdot \zeta_s \,d\fm \\
&= \lim_{\ve \to 0} \frac{1}{\ve} \int_s^{s+\ve} \int_M
 d\zeta_s \big( \nabla^{\Grad u_{s+r}}[ \widehat{P}^{\Grad u}_{s+r,t}(\phi_{s+\ve})] \big) \,d\fm \,dr
 +\int_M \widehat{P}^{\Grad u}_{s,t}(\del_s \phi_s) \cdot \zeta_s \,d\fm \\
&= \int_M d\zeta_s \big( \nabla^{\Grad u_s}[ \widehat{P}^{\Grad u}_{s,t}(\phi_s)] \big) \,d\fm \,dr
 +\int_M \widehat{P}^{\Grad u}_{s,t}(\del_s \phi_s) \cdot \zeta_s \,d\fm
\end{align*}
for almost every $s$.
We shall show that the RHS of \eqref{eq:key-} is nonnegative.

We first calculate by using \eqref{eq:delF} and $c'_{\alpha}=2(1-Kc_{\alpha})$ as
\[ \del_s \zeta_s =
 \frac{1}{\zeta_s} \Big\{ \scN(u_s)\scN'(u_s) \Lap u_s +\big( Kc_{\alpha}(t-s)-1 \big) F^2(\Grad u_s)
 +c_{\alpha}(t-s) d(\Lap u_s)(\Grad u_s) \Big\}. \]
Next, we have
\[ \nabla^{\Grad u_s} \zeta_s
 = \frac{1}{\zeta_s} \bigg\{ \scN(u_s) \scN'(u_s) \Grad u_s
 +\frac{c_{\alpha}(t-s)}{2} \nabla^{\Grad u_s}[F^2(\Grad u_s)] \bigg\}. \]
Hence
\begin{align*}
\Delta\!^{\Grad u_s} \zeta_s
&= \frac{\scN(u_s) \scN'(u_s)}{\zeta_s} \Lap u_s
 +\frac{\scN'(u_s)^2 -1}{\zeta_s} F^2(\Grad u_s)
 -\frac{\scN(u_s) \scN'(u_s)}{\zeta_s^2} d\zeta_s(\Grad u_s) \\
&\quad +\frac{c_{\alpha}(t-s)}{2\zeta_s} \Delta\!^{\Grad u_s}[F^2(\Grad u_s)]
 -\frac{c_{\alpha}(t-s)}{2\zeta_s^2} d\zeta_s \big( \nabla^{\Grad u_s}[F^2(\Grad u_s)] \big),
\end{align*}
where we used $\scN''=-1/\scN$ and $\Delta\!^{\Grad u_s}[F^2(\Grad u_s)]$
is understood in the weak sense.

Now we apply the improved Bochner inequality (Corollary~\ref{cr:Boc+}) to obtain
\begin{align*}
\Delta\!^{\Grad u_s} \zeta_s -\del_s \zeta_s
&= \frac{\scN'(u_s)^2 -Kc_{\alpha}(t-s)}{\zeta_s} F^2(\Grad u_s) \\
&\quad +\frac{c_{\alpha}(t-s)}{\zeta_s}
 \bigg\{ \Delta\!^{\Grad u_s}\bigg[ \frac{F^2(\Grad u_s)}{2} \bigg] -d(\Lap u_s)(\Grad u_s) \bigg\} \\
&\quad -\frac{\scN(u_s) \scN'(u_s)}{\zeta_s^2} d\zeta_s(\Grad u_s)
 -\frac{c_{\alpha}(t-s)}{2\zeta_s^2} d\zeta_s \big( \nabla^{\Grad u_s}[F^2(\Grad u_s)] \big) \\
&\ge \frac{\scN'(u_s)^2}{\zeta_s} F^2(\Grad u_s)
 +\frac{c_{\alpha}(t-s)}{\zeta_s F^2(\Grad u_s)} d\bigg[ \frac{F^2(\Grad u_s)}{2} \bigg]
 \bigg( \nabla^{\Grad u_s}\bigg[ \frac{F^2(\Grad u_s)}{2} \bigg] \bigg) \\
&\quad -\frac{\scN(u_s) \scN'(u_s)}{\zeta_s^2} d\zeta_s(\Grad u_s)
 -\frac{c_{\alpha}(t-s)}{2\zeta_s^2} d\zeta_s \big( \nabla^{\Grad u_s}[F^2(\Grad u_s)] \big)
\end{align*}
in the weak sense.
Substituting
%\begin{equation}\label{eq:zeta}
\[ d\zeta_s=\frac{1}{\zeta_s}\bigg\{ \scN(u_s) \scN'(u_s) du_s
 +\frac{c_{\alpha}(t-s)}{2} d[F^2(\Grad u_s)] \bigg\} \]
%\end{equation}
and recalling \eqref{eq:f1f2}, we obtain
\begin{align*}
\Delta\!^{\Grad u_s} \zeta_s -\del_s \zeta_s
&\ge \frac{\zeta_s^2 \scN'(u_s)^2 -\scN^2(u_s) \scN'(u_s)^2}{\zeta_s^3} F^2(\Grad u_s) \\
&\quad -\frac{c_{\alpha}(t-s) \scN(u_s) \scN'(u_s)}{\zeta_s^3}
 du_s \big( \nabla^{\Grad u_s}[F^2(\Grad u_s)] \big) \\
&\quad +\frac{c_{\alpha}(t-s)}{\zeta_s^3} \bigg\{ \frac{\zeta_s^2}{F^2(\Grad u_s)} -c_{\alpha}(t-s) \bigg\}
 d\bigg[ \frac{F^2(\Grad u_s)}{2} \bigg]
 \bigg( \nabla^{\Grad u_s}\bigg[ \frac{F^2(\Grad u_s)}{2} \bigg] \bigg) \\
&= \frac{c_{\alpha}(t-s) \scN'(u_s)^2}{\zeta_s^3} F^4(\Grad u_s) \\
&\quad -\frac{c_{\alpha}(t-s) \scN(u_s) \scN'(u_s)}{\zeta_s^3}
 du_s \big( \nabla^{\Grad u_s}[F^2(\Grad u_s)] \big) \\
&\quad +\frac{c_{\alpha}(t-s)}{\zeta_s^3} \frac{\scN^2(u_s)}{F^2(\Grad u_s)}
 d\bigg[ \frac{F^2(\Grad u_s)}{2} \bigg]
 \bigg( \nabla^{\Grad u_s}\bigg[ \frac{F^2(\Grad u_s)}{2} \bigg] \bigg).
\end{align*}
Since the Cauchy--Schwarz inequality for $g_{\Grad u_s}$ yields
\[ \big| du_s \big( \nabla^{\Grad u_s}[F^2(\Grad u_s)] \big) \big|
 \le F(\Grad u_s) \sqrt{d[F^2(\Grad u_s)] \big( \nabla^{\Grad u_s} [F^2(\Grad u_s)] \big)}, \]
we conclude that
\begin{align*}
&\Delta\!^{\Grad u_s} \zeta_s -\del_s \zeta_s \\
&\ge \frac{c_{\alpha}(t-s) \scN'(u_s)^2}{\zeta_s^3} F^4(\Grad u_s) \\
&\quad -\frac{c_{\alpha}(t-s) \scN(u_s) |\scN'(u_s)|}{\zeta_s^3}
 F(\Grad u_s) \sqrt{d[F^2(\Grad u_s)] \big( \nabla^{\Grad u_s} [F^2(\Grad u_s)] \big)} \\
&\quad +\frac{c_{\alpha}(t-s)}{\zeta_s^3} \frac{\scN^2(u_s)}{F^2(\Grad u_s)}
 d\bigg[ \frac{F^2(\Grad u_s)}{2} \bigg]
 \bigg( \nabla^{\Grad u_s}\bigg[ \frac{F^2(\Grad u_s)}{2} \bigg] \bigg) \\
&= \frac{c_{\alpha}(t-s)}{\zeta_s^3}
 \bigg( |\scN'(u_s)| F^2(\Grad u_s) -\frac{\scN(u_s)}{2F(\Grad u_s)}
 \sqrt{d[F^2(\Grad u_s)] \big( \nabla^{\Grad u_s} [F^2(\Grad u_s)] \big)} \bigg)^2 \\
& \ge 0
\end{align*}
in the weak sense.
Notice that, similarly to the proof of Theorem~\ref{th:L1},
we can take test functions from $H^1_0(M) \cap L^{\infty}(M)$ by virtue of \eqref{eq:hypo}.
Therefore the RHS of \eqref{eq:key-} is nonnegative and this completes the proof.
%%%%%%%%%%%%%
$\qedd$
\end{proof}

When $K>0$, choosing $\alpha=K^{-1}$ and letting $t \to \infty$ in \eqref{eq:key}
yields the following.

\begin{corollary}\label{cr:key}
Assume that $(M,F,\fm)$ is complete and satisfies
$\Ric_{\infty} \ge K>0$, $\sC_F<\infty$, $\sS_F<\infty$ and $\fm(M)=1$.
Then, for any $u \in \cC_c^{\infty}(M)$ with $0 \le u \le 1$ and satisfying \eqref{eq:hypo},
we have
\begin{equation}\label{eq:key'}
\sqrt{K} \scN\bigg( \int_M u \,d\fm \bigg)
 \le \int_M \sqrt{K\scN^2(u) +F^2(\Grad u)} \,d\fm.
\end{equation}
\end{corollary}

\begin{proof}
Let $(u_t)_{t \ge 0}$ be the global solution to the heat equation with $u_0=u$.
Taking $\alpha=K^{-1}$, we find $c_{\alpha} \equiv K^{-1}$ and hence by \eqref{eq:key}
\[ \sqrt{K \scN^2(u_t)} \le \sqrt{K \scN^2(u_t) +F^2(\Grad u_t)}
 \le P^{\Grad u}_{0,t}\Big( \sqrt{K\scN^2(u)+F^2(\Grad u)} \Big). \]
Letting $t \to \infty$, we deduce from the ergodicity (Proposition~\ref{pr:ergo}) that
\begin{align*}
u_t &\to \int_M u \,d\fm,\\
P^{\Grad u}_{0,t}\Big( \sqrt{K\scN^2(u)+F^2(\Grad u)} \Big)
 &\to \int_M \sqrt{K\scN^2(u)+F^2(\Grad u)} \,d\fm
\end{align*}
in $L^2(M)$.
Thereby we obtain \eqref{eq:key'}.
$\qedd$
\end{proof}

\subsection{Proof of Theorem~\ref{th:BL}}%%%%%%%%%%%%
%%%%%%%%%%%%%%%%%%%%%%%

\begin{proof}
Let $\theta \in (0,1)$.
Fix a closed set $A \subset M$ with $\fm(A)=\theta$ and consider
\[ u^{\ve}(x):=\max\{1-\ve^{-1}d(x,A),0\}, \qquad \ve>0. \]
Note that $F(\Grad u^{\ve})=\ve^{-1}$ on $B^-(A,\ve) \setminus A$, where
$B^-(A,\ve):=\{ x \in M \,|\, \inf_{y \in A} d(x,y)<\ve \}$
is the backward $\ve$-neighborhood of $A$.
Applying \eqref{eq:key'} to (smooth approximations of) $u^{\ve}$ and letting $\ve \downarrow 0$ implies,
with the help of $\scN(0)=\scN(1)=0$,
\[ \sqrt{K} \scN(\theta)
 \le \liminf_{\ve \downarrow 0} \frac{\fm(B^-(A,\ve))-\fm(A)}{\ve}. \]
This is the desired isoperimetric inequality for the reverse Finsler structure $\rev{F}$
(recall Definition~\ref{df:rev}) since, with $c:=\varphi^{-1}(\theta)/\sqrt{K}$,
\[ \sqrt{K} \scN(\theta) =\sqrt{\frac{K}{2\pi}} \e^{-Kc^2/2}, \qquad
 \theta =\varphi(\sqrt{K} c)
% =\frac{1}{\sqrt{2\pi}} \int_{-\infty}^{\sqrt{K} c} \e^{-b^2/2} \,db
 =\sqrt{\frac{K}{2\pi}} \int_{-\infty}^c \e^{-Ka^2/2} \,da. \]
Because the curvature bound $\Ric_{\infty} \ge K$ is common to $F$ and $\rev{F}$,
we also obtain \eqref{eq:BLisop}.
$\qedd$
\end{proof}

%%%%%%%%%%%%%%%
\if0%%%%%%%%%%%%%%%
%%%%%%%%%%%%%%%

The same argument as \cite[Corollary~7.5]{Oneedle} gives
the following corollary concerning normed spaces.
Even this simple case seems new.

\begin{corollary}[Isoperimetric inequality on normed spaces]\label{cr:norm}
Let $n \ge 2$ and $\|\cdot\|:\R^n \lra [0,\infty)$ be a continuous function satisfying$:$
\begin{enumerate}[{\rm (1)}]
\item $\|x\| >0$ for all $x \in \R^n \setminus \{(0,0,\ldots,0)\};$
\item $\|cx\|=c\|x\|$ for all $x \in \R^n$ and $c>0;$
\item $\|x+y\| \le \|x\|+\|y\|$ for any $x,y \in \R^n$.
\end{enumerate}
Consider the distance function $d(x,y):=\|y-x\|$ of $\R^n$,
and take a probability measure $d\fm=\e^{-\Phi}dx^1 dx^2 \cdots dx^n$ on $\R^n$
such that $dx^1 dx^2 \cdots dx^n$ is the Lebesgue measure
and $\Phi$ is a continuous function.

If $\Phi$ is $K$-convex with $K>0$ in the sense that
\[ \Phi\big( (1-\lambda)x+\lambda y \big) \le
 (1-\lambda) \Phi(x) +\lambda \Phi(y) -\frac{K}{2}(1-\lambda) \lambda d^2(x,y) \]
for all $x,y \in \R^n$ and $\lambda \in (0,1)$, then we have
\[ \cI_{(\R^n,d,\fm)}(\theta) \ge \cI_K(\theta) \quad \text{for all} \,\ \theta \in (0,1). \]
\end{corollary}

We remark that the completeness is clear in this case,
and $\sS_F<\infty$ is enjoyed for smooth approximations of the norm $\|\cdot\|$.

%%%%%%%%%%%%%%%
\fi%%%%%%%%%%%%%%%
%%%%%%%%%%%%%%%

\renewcommand{\refname}{{\large References}}
{\small%%%

}

\end{document}